\documentclass{article}
\usepackage{amsmath,amsfonts,amssymb,amsthm}
\usepackage{epsfig}
\usepackage{mathrsfs}
\usepackage{enumerate}
\newtheorem{theorem}{Theorem}

\newtheorem{koro}[theorem]{Corollary}

\newcommand{\bd}{{\mathrm{bd}}\,}
\newcommand{\inte}{{\mathrm{int}}\,}
\newcommand{\relint}{{\mathrm{relint}}\,}

\newcommand{\Nor}{{\mathrm{Nor}}\,}

\newcommand{\R}{{\mathbb R}}

\newcommand{\E}{{\mathrm e}}

\newcommand{\Q}{{\mathbb Q}}

\catcode`@=11
\def\section{%
\setcounter{equation}{0} \setcounter{theorem}{0} \@startsection
{section}{1}{\z@}{-4.0ex plus -1ex minus
    -.2ex}{2.3ex plus .2ex}{\bf}}
\catcode`@=12 \theoremstyle{definition}
\newtheorem{example}{Example}

\begin{document}

\title{Integral geometry of translation invariant functionals, I: The polytopal case}

\author{
Wolfgang Weil\\
 \\ {\it Department of Mathematics, 
 Karlsruhe Institute of Technology}\\
 {\it 76128 Karlsruhe, Germany}}

\date{}

\maketitle
\begin{abstract}
\noindent We study translative integral formulas for certain translation invariant functionals on convex polytopes and discuss local extensions and applications to Poisson processes and Boolean models.
\end{abstract}

\section{Introduction}

Classical integral geometry in the sense of Wilhelm Blaschke and his school deals with kinematic integral formulas for intrinsic volumes of convex bodies (see \cite{SW} for examples of such formulas and \cite{SW+} for historical remarks). More recent generalizations concern extensions to various set classes (sets of positive reach, polyconvex sets), local versions (for curvature measures), and variants for other transformation groups (translative integral geometry). In particular, the translative integral formulas for curvature measures of polyconvex sets (finite unions of convex bodies) found important applications in stochastic geometry, namely in the investigation of particle processes and Boolean models without any invariance assumptions. Intrinsic volumes and curvature measures are examples of additive, translation invariant and continuous (real- or measure-valued) functionals on the class ${\cal K}$ of convex bodies in $\R^d$ and the additivity is important for the extension of the functionals and their integral geometric formulas to unions of sets. The additivity is also a necessary requirement for kinematic formulas if they are proved using Hadwiger's celebrated characterization theorem. Therefore, the impression arose that, besides translation invariance and continuity, the additivity is an essential property of functionals like the intrinsic volumes in order to guarantee the validity of translative or kinematic formulas in integral geometry.

Surprisingly, as we shall show, the additivity is not crucial here. Instead of, the fact that intrinsic volumes have a local extension (given by the curvature measures) is of importance. It is the first goal of this work to show that translation invariant functionals $\varphi$ on the class ${\cal P}\subset {\cal K}$ of polytopes that have a local extension $\Phi$ (detailed definitions and formulations will be given in the next section) satisfy a translative integral formula (both, in a global and a local version) which is similar to the one for intrinsic volumes. 

As it turns out, functionals on ${\cal P}$ behave differently from those defined on ${\cal K}$. Namely, any weakly continuous, translation invariant functional $\varphi : {\cal P}\to \R$ which is additive has a local extension. On the other hand, there are many translation invariant  functionals $\varphi$ on ${\cal P}$ which have a local extension but are not additive. For these, the translative integral formulas which we shall prove open the way for applications to Poisson processes and Boolean models with polytopal grains. Such applications are given in the final two sections of this paper.  A particular example of a functional on polytopes $P$ which has a local extension, but is not additive, is given by the total ($k$-dimensional) content of the $k$-skeleton of $P$, $k\in \{ 0,...,d-2\}$. As it turns out, local extensions of translation invariant functionals are not unique (this is even true for the intrinsic volumes). Examples of this kind (on the level of polytopes) will be given at the end of Section 3. It remains an open problem to describe all local extension of a given functional $\varphi$.

Whereas we concentrate on the polytopal case here, functionals on ${\cal K}$ with a local extension will be discussed in the second part of this work. Then, it is natural to impose an additional continuity property. As it turns out, such continuous local functionals on ${\cal K}$ are automatically additive. For continuous local (and thus additive) functionals on ${\cal K}$ we will also show a translative integral formula. It will even hold for arbitrary additive functionals and generalizes results obtained in \cite[Section 11.1]{SW}.

\section{Definitions and results}

Let $\cal K$ be the space of  convex bodies in $\R^d$ (non-empty compact convex sets), supplied with the Hausdorff metric, and let $\cal P$ be the dense subset of convex polytopes. For notions from the theory of convex bodies, we refer to \cite{S}. Throughout this first part, we consider  real functionals $\varphi$ on ${\cal P}$ which are translation invariant. Such a functional $\varphi$ is {\it weakly continuous}, if it is continuous with respect to parallel displacements of the facets of the polytopes. 

We call a functional $\varphi : {\cal P}\to \R$ {\it local}, if it has a {\it local extension}  $\Phi :  {\cal P}\times {\cal B} \to \R$, namely a (measurable) kernel, meaning that $\Phi (\cdot ,A)$ is a measurable function on ${\cal P}$ for each  $A\in{\cal B}$ and $\Phi (P,\cdot)$ is a finite signed Borel measure on ${\R}^d$ for each $P\in {\cal P}$ (here ${\cal B}$ denotes the $\sigma$-algebra of Borel sets in $\R^d$), and such that $\Phi$ has the following properties:
\begin{itemize}
\item $\varphi (P) = \Phi(P,{\mathbb R}^d )$ for all $P\in {\cal P}$,
\item $\Phi$ is {\it translation covariant}, that is, satisfies $\Phi (P+x,A+x) = \Phi (P,A)$ for $P\in{\cal P}$, $A\in{\cal B}$, $x\in {\mathbb R}^d$,
\item $\Phi$ is {\it locally determined}, that is, $\Phi (P,A)=\Phi (Q,A)$ for $P,Q\in{\cal P}$, $A\in{\cal B}$, if there is an open set $U\subset {\mathbb R}^d$ with $P\cap U=Q\cap U$ and $A\subset U$.
\end{itemize}
By definition, a local functional $\varphi$ on $\cal P$ is measurable and translation invariant. 
For  $k\in\{ 0,...,d \}$, we say that $\varphi$ is {\it $k$-homogeneous}, if $\varphi (\alpha P) = \alpha^k \varphi (P)$ holds for all $\alpha\ge 0$ and all $P\in{\cal P}$. Similarly, for a measure-valued functional $\Phi$ on $\cal P$ (like a local extension of $\varphi$), $k$-homogeneity means that $\Phi (\alpha P, \alpha A) = \alpha^k \Phi (P,A)$ holds for all $\alpha \ge 0$, all $P\in{\cal P}$ and all Borel sets $A\subset \R^d$.

We now formulate some of the results of the paper. We recall that $V_j(K)$ denotes the $j$th intrinsic volume of a convex body $K$, $j=0,...,d$; thus $V_d(K)=\lambda (K)$ is the volume of $K$. Here, we also used $\lambda$ to denote the Lebesgue measure in $\R^d$. For a polytope $P$, let ${\cal F}_j(P)$ be the collection of $j$-faces of $P$, $j=0,...,d-1$, and let $n(P,F)$, for a face $F$ of $P$, be the intersection of the normal cone $N(P,F)$ of $P$ at $F$ with the unit sphere $S^{d-1}$ (this is a $(d-j-1)$-dimensional spherical polytope). Let ${\wp}_{d-j-1}^{d-1}$ be the class of $(d-j-1)$-dimensional spherical polytopes (again supplied with the topology of the Hausdorff metric).  For a $j$-dimensional face $F$ of a polytope $P$, let $\lambda_F$ be the restriction to $F$ of the ($j$-dimensional) Lebesgue measure in the affine hull of $F$. We also put ${\cal F}_d(P) = \{ P\}$, if $P$ is $d$-dimensional, and ${\cal F}_d(P) = \emptyset$ otherwise.

\begin{theorem}\label{th1} Let $\varphi$ be a local functional on $\cal P$ with local extension $\Phi$. Then $\varphi$ has a unique representation
\begin{equation}\label{func1}
\varphi (P) = \sum_{j=0}^{d-1} \varphi^{(j)}(P) + c_dV_d(P)
\end{equation}
with $j$-homogeneous local functionals $\varphi^{(j)}$ on $\cal P$ and a constant $c_d\in\R$. Also, $\Phi$ has a unique representation
\begin{equation}\label{func2}
\Phi (P,\cdot) = \sum_{j=0}^{d-1} \Phi^{(j)}(P,\cdot) + c_d\lambda (P\cap\cdot)
\end{equation}
with $j$-homogeneous measure-valued functionals $\Phi^{(j)}$ on $\cal P$ and the same constant $c_d$. 
Each $\Phi^{(j)}$ has the form
\begin{equation}\label{func3}
\Phi^{(j)} (P,\cdot) = \sum_{F\in {\cal F}_j(P)} f_j(n(P,F))\lambda_F
\end{equation}
with a (uniquely determined) measurable function $f_j$ on ${\wp}_{d-j-1}^{d-1}$. We put $f_d=c_d$ and call $f_0,...,f_d$ the {\em associated functions} of $\Phi$. 

For $j=0,...,d-1$, the kernel $\Phi^{(j)}$ is a local extension of $\varphi^{(j)}$ and so
\begin{align}\label{func4}
\varphi^{(j)} (P) &= \sum_{F\in {\cal F}_j(P)} f_j(n(P,F)) V_j(F).
\end{align}

If $\Phi\ge 0$, then $f_j\ge 0, j=0,...,d$, and thus $\Phi^{(j)}\ge 0, \varphi^{(j)}\ge 0$, for $j=0,...,d-1$, and $\varphi\ge 0$.
\end{theorem}

For a local functional $\varphi$ on ${\cal P}$, there is a natural decomposition $\varphi=\varphi^+ -\varphi^-$ with local functionals $\varphi^+, \varphi^-\ge 0$. Namely, the local extension $\Phi$ of $\varphi$ admits a Hahn-Jordan decomposition $\Phi(P,\cdot)=\Phi^+(P,\cdot) -\Phi^-(P,\cdot)$, for each $P\in{\cal P}$, into (nonnegative) measures $\Phi^+(P,\cdot), \Phi^-(P,\cdot)$. It follows from the construction of the Hahn-Jordan decomposition, but also from the explicit representations \eqref{func2} and \eqref{func3}, that $P\mapsto \Phi^+(P,\cdot)$ and $P\mapsto \Phi^-(P,\cdot)$ are measurable and that the kernels $\Phi^+, \Phi^-$ are translation covariant and locally defined. Hence, $ \varphi^+(P)=\Phi^+(P,\R^d)$ and $ \varphi^-(P)=\Phi^-(P,\R^d)$ define local functionals $\varphi^+, \varphi^-\ge 0$ on $\cal P$ with $\varphi=\varphi^+ -\varphi^-$. The decomposition $\Phi(P,\cdot)=\Phi^+(P,\cdot) -\Phi^-(P,\cdot)$ corresponds to the Hahn-Jordan decompositions  $\Phi^{(j)}(P,\cdot)=\Phi^{(j)+}(P,\cdot) -\Phi^{(j)-}(P,\cdot)$ of the $j$-homogeneous parts, and the latter is equivalent to the decomposition of the associated function  $f_j=f_j^+-f_j^-$ into its positive and negative parts, $j=0,...,d-1$. We may, therefore, assume in the following that the local functional $\varphi$ under consideration is nonnegative, if this is helpful.

However, as we already mentioned, the local extension $\Phi$ of a local functional $\varphi$ need not be unique. Even more, for each $j=1,...,d-1$, there are $j$-homogeneous kernels $\Phi^{(j)}\ge 0$, $\tilde\Phi^{(j)}\ge 0$, with $\Phi^{(j)}\not = \tilde\Phi^{(j)}$ which are local extensions of the same local functional  $\varphi^{(j)} = \Phi^{(j)}(\cdot,\R^d) = \tilde\Phi^{(j)}(\cdot ,\R^d)$. This non-uniqueness can also affect the Hahn-Jordan decomposition just explained and therefore the decomposition $\varphi=\varphi^+ -\varphi^-$ of a local functional $\varphi$ as a difference of nonnegative local functionals may depend on the choice of the local extension $\Phi$.

\section{Properties of local functionals} 

In this section, we first give the proof of Theorem \ref{th1}.

\begin{proof}
Let $\varphi$ be a local functional on ${\cal P}$ and let $\Phi$ be a local extension of $\varphi$. In the following, we use  arguments from the first part of the proof of \cite[Theorem 3.1]{KW99}. Namely, for $K\in {\cal P}$ and a Borel set $A\in{\cal B}$, we have
$$
A = (A\setminus P) \cup\bigcup_{j=0}^d\bigcup_{F\in{\cal F}_j(P)} (A\cap \relint F) ,
$$
where the sets in this decomposition are mutually disjoint. Hence
$$
\Phi(P,A) = \Phi(P,A\setminus P) +\sum_{j=0}^d\sum_{F\in{\cal F}_j(P)} \Phi(P,A\cap \relint F).
$$
As in \cite[p. 124]{Schn78}, the fact that $\Phi$ is locally defined, implies $\Phi(P,A\setminus P)=0$. Also, if $A\subset\relint F$, $F\in{\cal F}_j(P)$, then the translation covariance of $\Phi$ implies 
\begin{equation}\label{local}
\Phi (P,A) = b_j(P,F)\lambda_F(A)
\end{equation}
with a constant $b_j(P,F)\in\R$ which, for $j=0,...,d-1$, depends only on the normal cone $N(P,F)$ of $P$ at $F$ (see \cite[p. 123]{Schn78}, for a similar argument). If $j=d$, then $F=P$ and $b_d(P,P)=c_d$ is a constant independent of $P$. For $j\le d-1$, the constants $b_j(P,F)$ give rise to a function $f_j$ on ${\wp}^{d-1}_{d-j-1}$ through $f_j(p) = f_j(n(P,F)) = b_j(P,F)$ (with $p=n(P,F)$), since every spherical polytope $p\in {\wp}^{d-1}_{d-j-1}$ is generated by the normal cone $N(P,F)$ of a face $F\in{\cal F}_j(P)$ of a suitable polytope $P$. To be more precise, let $p\in {\wp}^{d-1}_{d-j-1}$ be given and let $\tilde P$ be the intersection of all closed halfspaces in $\R^d$ with the origin $0$ in the boundary and with outer normal in $p$. Then, $\tilde P$ is a polyhedral set and $L= p^\bot$ is a $j$-dimensional face of $\tilde P$. Intersecting $\tilde P$ with the unit cube $W=[-1/2,1/2]^d$, we obtain a polytope $P$ which has $F=L\cap W$ as a $j$-face and satisfies $p= n(P,F)$.

Equation \eqref{local} also yields the measurability of $f_j$. Namely, let ${\cal Q}$ be the set of all polytopes $P\subset \R^d$, for which $0$ is in the relative interior of some $j$-face $F$ of $P$ and such that the unit cube $W = [-1/2,1/2]^d$ intersects $F$ only in relative interior points and does not contain points of any other $j$-face $G$ of $P$. Let ${\cal M}$ consist of all intersections $P\cap W, P\in{\cal Q}$. It is easy to see that ${\cal M}$ is a measurable subset of ${\cal P}$ (which is neither open nor closed). Also, it follows that the mapping $P\mapsto n(P,F)$ is a bi-continuous bijection from ${\cal M}$ to ${\wp}_{d-j-1}^{d-1}$. Let $A = rB$ with $r<1$, then 
$$\Phi (P,A) = b_j(P,F)\lambda_F(A) = f_j(n(P,F)) c(j,r),$$ 
for $P\in{\cal M}$, where the constant $c(j,r)$ is the $j$-volume of a $j$-dimensional ball of radius $r$. Since $P\mapsto \Phi(P,A)$ is measurable, $f_j$ is measurable as the composition of a continuous and a measurable mapping.

We now define a signed measure $\Phi^{(j)}(P,\cdot)$ by
\begin{equation}\label{def}
\Phi^{(j)} (P,\cdot) = \sum_{F\in {\cal F}_j(P)} f_j(n(P,F)) \lambda_F
\end{equation}
(which corresponds to \eqref{func3}) and obtain the decomposition
$$\Phi (P,A) = \sum_{j=0}^{d-1} \Phi^{(j)}(P,A) + c_d\lambda (P\cap A)  $$
which corresponds to \eqref{func2}.
Since $n(\alpha P,\alpha F) = n(P,F)$, for $\alpha > 0$, we have
$$
\Phi^{(j)} (\alpha P,\alpha A) = \sum_{F\in {\cal F}_j(P)} f_j(n(P,F)) \lambda_{\alpha F}(\alpha A) = \alpha^j\Phi^{(j)} (P,A)
$$
and get
\begin{equation}\label{expansion}
\Phi (\alpha P,\alpha A) = \sum_{j=0}^{d-1} \alpha^j \Phi^{(j)}(P,A) + \alpha^d c_d\lambda (P\cap A).
\end{equation}
The polynomial expansion \eqref{expansion} shows that the decomposition of $\Phi$ into $j$-homogeneous parts is unique and that the $j$-homogeneous parts $\Phi^{(j)}$ of $\Phi$ necessarily have the form \eqref{def}.
Also, by construction, the functions $f_0,...,f_{d-1}$ are uniquely determined by $\Phi$. 
Putting $\varphi^{(j)}(P) =  \Phi^{(j)} (P,\R^d)$, the formulas \eqref{func1} and \eqref{func4} result. 

If $\Phi\ge 0$, then \eqref{local} shows that $f_j\ge 0, j=0,...,d$, and $\Phi^{(j)}\ge 0, \varphi^{(j)}\ge 0$, and $\varphi\ge 0$ follow.
\end{proof}

Concerning the non-uniqueness of local extensions on $\cal P$, we give the following example.

\begin{example} For $j\in\{ 1,...,d-1\}$, we consider a $j$-homogeneous kernel $\Phi^{(j)}\ge 0$ on ${\cal P}$ given by
$$
\Phi^{(j)} ( P,\cdot) = \sum_{F\in {\cal F}_j(P)} f_j(n(P,F)) \lambda_{F}
$$
for some measurable function $f_j\ge 0$ on ${\wp}^{d-1}_{d-j-1}$. For $x_0\in\R^d\setminus \{ 0\}$, we define $\tilde \Phi^{(j)}$ by
\begin{equation}\label{ex}
\tilde \Phi^{(j)}(P,\cdot ) = \sum_{F\in {\cal F}_j(P)} \left( f_j(n(P,F)) +\int_{n(P,F)} \langle u,x_0\rangle \omega_{d-j-1}( du)\right) \lambda_{F},
\end{equation}
for $P\in{\cal P}$. Here, $\omega_{d-j-1}$ denotes the spherical Lebesgue measure in the subspace generated by $n(P,F)$, namely ${F^\bot}$. Using the well-known representation of the $j$th support measure $\Lambda_j(P,\cdot)$ of $P$ and its connection to the $j$th surface area measure (see \cite[(14.11)]{SW}, where different normalizations are used), we can re-write \eqref{ex} as
$$
\tilde \Phi^{(j)}(P,A) =  \Phi^{(j)}(P,A)+\int_{\Nor P} {\bf 1}_A(x)\langle u,x_0\rangle \Lambda_j(P, d(x,u)),
$$
for $A\in{\cal B}$, were the integration is over the (generalized) normal bundle $\Nor P$ of $P$.

Obviously, $\Phi^{(j)} \not= \tilde\Phi^{(j)}$, but
\begin{align*}
\tilde\Phi^{(j)} (P,\R^d) &= \Phi^{(j)} (P,\R^d)+\int_{\Nor P} \langle u,x_0\rangle \Lambda_j(P, d(x,u))\cr
&= \Phi^{(j)} (P,\R^d)+ \int_{S^{d-1}} \langle u,x_0\rangle  \Psi_j(P, du)\cr
&= \Phi^{(j)} (P,\R^d) ,
\end{align*}
since the $j$th surface area measure $\Psi_j(P,\cdot )$ of $P$ has centroid $0$. Hence, $\Phi^{(j)}$ and $\tilde\Phi^{(j)}$ are local extensions of the same local functional $\varphi^{(j)}$. Moreover, if
$$f_j(p) +\int_{p} \langle u,x_0\rangle \omega_{d-j-1}( du) \ge 0, \quad p\in {\wp}^{d-1}_{d-j-1},
$$
then both local extensions $\Phi_j$ and $\tilde\Phi_j$ are nonnegative. For example, this is the case if $f_{j}(p) = \omega_{d-j-1}(p)$, and  $\| x_0\|\le 1$. Then $\Phi^{(j)}(P,\cdot)$ is the $j$th curvature measure of $P$ and thus $V_j$ has a second local extension $\tilde\Phi^{(j)}\ge 0$. Even more, in this case  both local extensions $P\mapsto\Phi_j(P,\cdot)$ and $P\mapsto\tilde\Phi^{(j)}(P,\cdot)$ of $V_j$ are additive.
\end{example}

\section{Additive functionals}

In this section, we consider functionals $\varphi$ on ${\cal P}$ which are additive hence {\it valuations} and weakly continuous. Here, weak continuity means that $\varphi (P)$ behaves continuously with respect to individual shifts of the hyperplanes generating the facets of $P$ (see \cite{McM83, McM93}, for the precise definition). Any functional $\varphi$ on ${\cal P}$ which is continuous with respect to the Hausdorff  metric is weakly continuous (see \cite{McM83, McM93} again).

For simplicity, and following
 \cite[Section 11.1]{SW}, we call a translation invariant, additive and weakly continuous functional on ${\cal P}$ a {\it standard functional}.

\begin{theorem}\label{th1'} Each standard functional $\varphi$ on $\cal P$ is local and has a local extension $\Phi (P,\cdot)$ which is additive as a measure-valued function of $P\in{\cal P}$. In that case, the $j$-homogeneous part $\varphi^{(j)}$ of $\varphi$ is a standard functional with the $j$-homogeneous part  $\Phi^{(j)}(P,\cdot)$ of $\Phi (P,\cdot)$ as local extension. Moreover, the measures  $\Phi^{(j)}(P,\cdot)$ are additive as measure-valued functions of $P\in{\cal P}$ and the associated functions $f_j$ are simply additive, $j=0,...,d$. 
\end{theorem}

\begin{proof} It follows from results of McMullen \cite{McM83, McM93} that a standard functional $\varphi$ on ${\cal P}$ allows a unique decomposition 
$$
\varphi = \sum_{j=0}^d \varphi^{(j)}
$$
into $j$-homogeneous standard functionals $\varphi^{(j)}$ which are of the form
$$
\varphi^{(j)} (P) = \sum_{F\in {\cal F}_j(P)} f_{j}(n(P,F))V_j(F)
$$
with a simply additive function $f_j$ on ${\wp}^{d-1}_{d-j-1}$. 

For $j=0,...,d$, we define a signed measure $\Phi^{(j)}(P,\cdot)$ by
$$
\Phi^{(j)} (P,\cdot) = \sum_{F\in {\cal F}_j(P)} f_j(n(P,F)) \lambda_F
$$
and put
$$\Phi (P,\cdot) = \sum_{j=0}^{d} \Phi^{(j)}(P,\cdot). $$
Then $\Phi^{(j)}$ is a local extension of $\varphi^{(j)}$ and $\Phi$ is a local extension of $\varphi$ and  $\Phi^{(j)} (P,\cdot)$ and $\Phi (P,\cdot)$ are additive in $P$. 
\end{proof}

\section{Translative integral formulas}

Let $\varphi$ be a local functional on ${\cal P}$ with local extension $\Phi$. Let $\varphi^{(j)}$ and $\Phi^{(j)}$ be the $j$-homogeneous parts of $\varphi$ and $\Phi$ with associated function $f_j$, $j=0,...,d$. In this section, we prove a translation formula for $\Phi^{(j)}$ and its iteration. For this purpose, we remark the following. Let $P,Q$ be polytopes, $F$ a $k$-face of $P$, $G$ a $(d+j-k)$-face of $Q$. If $F$ and $G$ are in general relative position (see \cite[p. 214]{SW}, for this notion) and if $x\in\R^d$ is such that $F$ and $G+x$ intersect in relative interior points, then $F\cap (G+x)$ is a $j$-face of $P\cap (Q+x)$. The normal cone of $F\cap (G+x)$ does not depend on the choice of $x$ and we denote its intersection with $S^{d-1}$ by $n(P,Q;F,G)$. In the same way, for polytopes $P_1,...,P_k$ with faces $F_1,...,F_k$ in general relative position, we define the mixed spherical polytope $n(P_1,...,P_k;F_1,...,F_k)$ recursively (see \cite[Section 6.4]{SW}, for details).

In the following, we often abbreviate the translation $A+x$ of a set $A\subset \R^d$ by $A^x$. We also use the determinant $[F_1,...,F_k]$ of faces $F_1,...,F_k$, as it is defined in \cite[p. 183]{SW}.

\begin{theorem}\label{trans1}
Let $\varphi$ be a local functional on ${\cal P}$ with local extension $\Phi$. Let $\varphi^{(j)}$ and $\Phi^{(j)}$ be the $j$-homogeneous parts of $\varphi$ and $\Phi$ with associated function $f_j$, $j=0,...,d$.
Then, for polytopes $P,Q \in {\cal P}$ and Borel sets  $A, B \in {\cal B}$, we have
\[ \int _{{\mathbb R}^d} \Phi^{(j)}(P \cap Q^x, A \cap B^x) \, \lambda(d x) =\sum_{m=j}^d \Phi_{m,d-m+j}^{(j)} (P,Q;A \times B) \]
with finite signed measures $\Phi_{m,d-m+j}^{(j)}(P,Q;\cdot)$ on ${\mathbb R}^d \times {\mathbb R}^d$, which are defined by
\[ \Phi_{m,d-m+j}^{(j)} (P,Q;\cdot) = \sum_{F \in {\cal F}_{m}(P)}  
\sum_{G \in {\cal F}_{d-m+j}(Q)} f_j(n(P,Q;F,G)) [F,G] \lambda_{F}
\otimes \lambda_{G} \]
$(m = j,\dots,d)$. In particular,
\begin{align*}
&  \Phi_{j,d}^{(j)} (P,Q;A \times B) = \Phi^{(j)}(P,A) \lambda (Q\cap B), \\
&  \Phi_{d,j}^{(j)}(P,Q;A \times B) = \lambda (P\cap A) \Phi^{(j)}(Q,B).
\end{align*}

More generally, for $k\ge 2$, polytopes $P_1,...,P_k \in {\cal P}$ and Borel sets  $A_1,...,A_k \in {\cal B}$, the following iterated translative integral formula holds,
\begin{align}\label{mixedint}
&  \int_{({\mathbb R}^d)^{k-1}}
\Phi^{(j)}(P_1\cap P_2^{x_2}\cap \dots \cap P_k^{x_k},A_1\cap A_2^{x_2}\cap \dots \cap A_k^{x_k})\, \lambda^{k-1} (d(x_2,\dots ,x_k)) \nonumber\\
& = \sum_{{m_1,\dots,m_k=j}\atop{m_1+\dots +m_k=(k-1)d+j}}^d
\Phi^{(j)}_{m_1,\dots,m_k}(P_1,\dots,P_k;A_1\times \dots \times A_k)
\end{align}
with mixed (signed) measures $\Phi^{(j)}_{m_1,\dots,m_k}(P_1,\dots,P_k;\cdot)$ on ${({\mathbb R}^d)^{k}}$ given by
\begin{align}\label{mixedpol}
\Phi_{m_1,\dots,m_k}^{(j)}(P_1,\dots,P_k;\cdot ) &=\sum_{F_1\in {\cal F}_{m_1}(P_1)}\dots \sum_{F_k\in
{\cal F}_{m_k}(P_k)} f_j(n(P_1,\dots,P_k;F_1,\dots,F_k))\nonumber\\
&  \quad \times \,[F_1,\dots,F_k] \lambda_{F_1}\otimes\cdots\otimes\lambda_{F_k}.
\end{align}

The measure $\Phi^{(j)}_{m_1,\dots,m_k}(P_1,\dots,P_k;\cdot)$ is homogeneous of degree $m_i$  in $P_i$, $i=1,...,k$.
\end{theorem}

\begin{proof} The proof of the results follows closely the arguments used in the proofs of Theorems 5.2.2 and 6.4.1 in \cite{SW}. We therefore present only the essential steps. In view of the decomposition $f_j =  f_j^+-f_j^-,$ with $f_j^+,f_j^-\ge 0$, we may assume $f_j\ge 0$.

For polytopes $P,Q \in {\cal P}$, Borel sets  $A, B \in {\cal B}$ and $x\in\R^d$, Theorem \ref{th1} implies
$$
\Phi^{(j)}(P \cap Q^x, A \cap B^x) = \sum_{F'\in {\cal F}_j(P\cap Q^x)} f_j(n(P\cap Q^x,F'))\lambda_{F'}(A\cap B^x) .
$$
Due to the arguments given in \cite[p. 184]{SW}, for $\lambda$-almost all $x$, the face $F'$ is the intersection $F'=F\cap G^x$ of some $m$-face $F$ of $P$ with a $(d-m+j)$-face $G$ of $Q$, $m\in\{ j,...,d\}$ (and such that $F$ and $G^x$ meet in relative interior points). Therefore, for $\lambda$-almost all $x$,
\begin{align*}
&\Phi^{(j)}(P \cap Q^x, A \cap B^x)\cr 
&\quad= \sum_{m=j}^d\sum_{F\in {\cal F}_m(P)} \sum_{G\in {\cal F}_{d-m+j}(Q)} f_j(n(P\cap Q^x,F\cap G^x))\lambda_{F\cap G^x}(A\cap B^x)\cr
&\quad= \sum_{m=j}^d\sum_{F\in {\cal F}_m(P)} \sum_{G\in {\cal F}_{d-m+j}(Q)} f_j(n(P,Q;F,G))\lambda_{F\cap G^x}(A\cap B^x) .
\end{align*}

Since $x\mapsto n(P \cap Q^x, F \cap G^x)$ and $x\mapsto\lambda_{F\cap G^x}(A\cap B^x)$ are measurable for $\lambda$-almost all $x$ (compare the corresponding results in \cite[Section 5.2]{SW}) and since $f_j$ is measurable, we obtain the measurability of
$$x\mapsto \Phi^{(j)}(P \cap Q^x, A \cap B^x)$$
for $\lambda$-almost all $x$. Notice that we cannot use Lemma 5.2.1 in \cite{SW} directly here, since we do not assume that $P\mapsto \Phi^{(j)}(P,\cdot )$ is weakly continuous. Hence, the following integral is well defined and we obtain
\begin{align*}
\int_{\R^d}&\Phi^{(j)}(P \cap Q^x, A \cap B^x)\lambda (dx)\cr 
&= \sum_{m=j}^d\sum_{F\in {\cal F}_m(P)} \sum_{G\in {\cal F}_{d-m+j}(Q)} f_j(n(P,Q;F,G))\int_{\R^d}\lambda_{F\cap G^x}(A\cap B^x) \lambda (dx).
\end{align*}
In \cite[pp. 185-186]{SW}, it was shown that 
$$\int_{\R^d}\lambda_{F\cap G^x}(A\cap B^x) \lambda (dx)= [F,G]\lambda_F(A)\lambda_G(B)
$$
and so the first part of the theorem follows. In particular, for $m=j$ we have
\begin{align*}&\sum_{F\in {\cal F}_j(P)} \sum_{G\in {\cal F}_{d}(Q)} f_j(n(P,Q;F,G))[F,G]\lambda_F(A)\lambda_G(B)\cr
&\quad
= \sum_{F\in {\cal F}_j(P)} f_j(n(P,F))\lambda_F(A)\lambda (Q\cap B)\cr
&\quad = \Phi^{(j)}(P,A)\lambda (Q\cap B)
\end{align*}
and for $m=d$ we get similarly
\begin{align*}&\sum_{F\in {\cal F}_d(P)} \sum_{G\in {\cal F}_{j}(Q)} f_j(n(P,Q;F,G))[F,G]\lambda_F(A)\lambda_G(B)\cr
&\quad = \sum_{G\in {\cal F}_j(Q)} f_j(n(Q,G))\lambda (P\cap A)\lambda_G(B)\cr
&\quad = \lambda (P\cap A)\Phi^{(j)}(Q,B).
\end{align*}

For the second part, the measurability of the integrand follows from the case $m=2$ by iteration, and in the same way we get also the iterated translation formula using a recursion formula for the determinant of faces (see \cite[p. 231]{SW}, for details). 

The homogeneity property of the mixed measures follows immediately from the explicit representation \eqref{mixedpol}.
\end{proof}

We remark that \eqref{mixedint}, for all Borel sets $A_1,...,A_k$, is equivalent to
\begin{align}\label{mixedint2}
&  \int_{({\mathbb R}^d)^{k-1}}\int_{\R^d} f(x_1,x_1-x_2,...,x_1-x_k) 
\Phi^{(j)}(P_1\cap P_2^{x_2}\cap \dots \cap P_k^{x_k},dx_1)\cr 
&\quad\times\lambda^{k-1} (d(x_2,\dots ,x_k)) \\
& = \sum_{{m_1,\dots,m_k=j}\atop{m_1+\dots +m_k=(k-1)d+j}}^d \int_{({\mathbb R}^d)^{k}} f(x_1,...,x_k) 
\Phi^{(j)}_{m_1,\dots,m_k}(P_1,\dots,P_k;d(x_1,...,x_k)),\nonumber
\end{align}
for all continuous functions $f$ on $({\mathbb R}^d)^{k}$ (compare \cite[formula (6.16)]{SW}).

\begin{koro}\label{trans2}
If $P, Q \in {\cal P}$ are
polytopes and $j \in \{0,\dots,d\}$, then
\begin{align*}
 \int_{{\mathbb R}^d}& \varphi^{(j)}(P \cap Q^x)\, \lambda(dx)\cr &\quad = \varphi^{(j)}(P)V_d(Q) + \sum_{m=j+1}^{d-1} \varphi_{m,d-m+j}^{(j)}  
(P,Q) + \varphi^{(j)}(Q)V_d (P),
\end{align*}
where
\begin{align*} &\varphi_{m,d-m+j}^{(j)}(P,Q)\cr
&\quad = \sum_{F \in {\cal F}_{m}(P)} \sum_{G \in {\cal F}_{d-m+j}(Q)} 
f_j(n(P,Q;F,G)) [F, G] V_m(F) V_{d-m+j}(G). 
\end{align*}

More generally, for $k\ge 2$ and polytopes $P_1,...,P_k \in {\cal P}$, we have 
\begin{align*}
  \int_{({\mathbb R}^d)^{k-1}}&
\varphi^{(j)}(P_1\cap P_2^{x_2}\cap \dots \cap P_k^{x_k})\, \lambda^{k-1} (d(x_2,\dots ,x_k)) \\
& \quad = \sum_{{m_1,\dots,m_k=j}\atop{m_1+\dots +m_k=(k-1)d+j}}^d
\varphi^{(j)}_{m_1,\dots,m_k}(P_1,\dots,P_k)
\end{align*}
with 
\begin{align*}
\varphi_{m_1,\dots,m_k}^{(j)}(P_1,\dots,P_k) &=\sum_{F_1\in {\cal F}_{m_1}(P_1)}\dots \sum_{F_k\in
{\cal F}_{m_k}(P_k)} f_j(n(P_1,\dots,P_k;F_1,\dots,F_k))\\
&  \quad \times \,[F_1,\dots,F_k] V_{m_1}(F_1)\cdots V_{m_k}(F_k).
\end{align*}

The mixed functional $\varphi^{(j)}_{m_1,\dots,m_k}(P_1,\dots,P_k)$ is homogeneous of degree $m_i$  in $P_i$, $i=1,...,k$.

Moreover,
$$P_i\mapsto \varphi^{(j)}_{m_1,\dots,m_i,\dots,m_k}(P_1,\dots,P_i,\dots, P_k)$$
is a local functional with local extension
$$(P_i,A_i)\mapsto \Phi^{(j)}_{m_1,\dots,m_i,\dots,m_k}(P_1,\dots,P_i,\dots, P_k,\R^d\times\cdots\times\R^d\times A_i\times\R^d\cdots\times\R^d) .$$
\end{koro}

\begin{proof}
It remains to prove the last assertion, namely the translation covariance and local determination of $$(P_i,A_i)\mapsto \Phi^{(j)}_{m_1,\dots,m_i,\dots,m_k}(P_1,\dots,P_i,\dots, P_k,\R^d\times\cdots\R^d\times A_i\times\R^d\cdots\times\R^d) .$$
Both properties are immediate consequences of \eqref{mixedpol}.
\end{proof}

We comment  shortly on some further properties of the mixed functionals \linebreak $\varphi^{(j)}_{m_1,\dots,m_k}$ and their local extensions   $\Phi^{(j)}_{m_1,\dots,m_k}$. For details, we refer to \cite[Section 6.4]{SW}, where corresponding results are discussed for the mixed measures and functionals of intrinsic volumes and curvature measures. Generally, the results follow from \eqref{mixedpol}. A first property is the {\it symmetry}. If the indices $m_1,\dots,m_k$, the polytopes $P_1,...,P_k$ and the Borel sets $A_1,...,A_k$ are interchanged by the same permutation, the values of the mixed measure $\Phi^{(j)}_{m_1,\dots,m_k}(P_1,...,P_k;A_1\times\cdots\times A_k)$ and of the mixed functional $\varphi^{(j)}_{m_1,\dots,m_k}(P_1,...,P_k)$ remain unchanged. The second property is the {\it decomposability}. If $m_k=d$ (and this case is sufficient due to the symmetry), then
\begin{align}\label{decomp}
&\Phi^{(j)}_{m_1,\dots,m_{k-1},d}(P_1,...,P_{k-1},P_k;A_1\times\cdots\times A_{k-1}\times A_k)\cr
&\quad = \Phi^{(j)}_{m_1,\dots,m_{k-1}}(P_1,...,P_{k-1};A_1\times\cdots\times A_{k-1})\lambda (P_k\cap A_k)
\end{align}
and
\begin{align*}
\varphi^{(j)}_{m_1,\dots,m_{k-1},d}(P_1,...,P_{k-1},P_k) = \varphi^{(j)}_{m_1,\dots,m_{k-1}}(P_1,...,P_{k-1})V_d(P_k).
\end{align*}
Besides the trivial case $j=d$ of Corollary \ref{trans2}, we also notice the simple case $j=d-1$, where we obtain
\begin{align*}
 \int_{{\mathbb R}^d}& \varphi^{(d-1)}(P \cap Q^x)\, \lambda(dx) = \varphi^{(d-1)}(P)V_d(Q) + \varphi^{(d-1)}(Q)V_d (P).
\end{align*}

As a first example, we mention the case $f_j=1, B=\R^d$. Then, Theorem \ref{trans1} implies a result for the $j$th Hausdorff measure ${\cal H}^j_{skel}$ (concentrated on the $j$-skeleton of the polytopes),
\[ \int _{{\mathbb R}^d} {\cal H}^j_{skel}(P \cap Q^x, A) \, \lambda(d x) =\sum_{m=j}^d \sum_{F \in {\cal F}_{m}(P)}  
\sum_{G \in {\cal F}_{d-m+j}(Q)} [F,G] \lambda_{F}(A) V_{d-m+j}(G). \]
In particular, for $j=0$ and $A=\R^d$, we obtain the obvious formula for the number $v_0$ of vertices,
\[ \int _{{\mathbb R}^d} v_0(P \cap Q^x) \, \lambda(d x) =\sum_{m=0}^d \sum_{F \in {\cal F}_{m}(P)}  
\sum_{G \in {\cal F}_{d-m}(Q)} [F,G] V_m(F) V_{d-m}(G). \]
We can also obtain a kinematic version of these results. Namely, if we integrate $[F,\vartheta G]$, for $F \in {\cal F}_{m}(P), 
G \in {\cal F}_{d-m+j}(Q)$, over all rotations $\vartheta$ with respect to the normalized Haar measure $\nu$ on the rotation group $SO_d$, we get
\begin{equation}\label{rotationformula}\int_{SO_d} [F,\vartheta G]\nu (d\vartheta ) = c_{j,d}^{m,d-m+j} ,\end{equation}
for example from Theorem 5.3.1 in \cite{SW} (with a constant $ c_{j,d}^{m,d-m+j}=c^d_jc^m_dc^{d-m+j}_d$ given explicitly by formula (5.5) in \cite{SW}).
Hence, the corresponding integration over the group $G_d$ of rigid motions (with invariant measure $\mu$) yields
\[ \int _{G_d} {\cal H}^j_{skel}(P \cap gQ) \, \mu(d g) =\sum_{m=j}^d c_{j,d}^{m,d-m+j}{\cal H}^m_{skel}(P){\cal H}^{d-m+j}_{skel}(Q) , \]
a result which can also be obtained by applying the principal kinematic formula to the various faces of $P$ and $Q$.

\section{$GP$-extensions}

Let ${\cal R}= U({\cal K})$ denote the class of polyconvex sets in $\R^d$ (finite unions of convex bodies) and let $U({\cal P})$ be the subclass of finite unions of polytopes. By Groemer's extension theorem (cf. \cite[Theorem 14.4.2]{SW}), every standard functional on ${\cal P}$ has an additive extension to $U({\cal P})$ (see \cite[Theorem 14.4.3]{SW}). For this extension, the inclusion-exclusion formula holds which shows that the extension is unique. In the following, we discuss similar extensions of local functionals to certain sets in $U({\cal P})$.

We say that polytopes $P,Q\in {\cal P}$ are in {\it mutual general position}, if either $P\cap Q=\emptyset$ or, for $j=0,...,d$, any $j$-face $F$ of $P\cap Q$ is the closure of the intersection of the relative interior of a $k$-face $G$ of $P$ and the relative interior of a $(d+j-k)$-face $H$ of $Q$, for some $k\in\{ j,...,d\}$. Notice that there is a basic difference between mutual general  position and the notion of general relative position which we used in the last section. For later use, we remark that, for $P,Q\in {\cal P}$, the polytopes $P, Q+x$ are in mutual general position, for $\lambda$-almost all $x\in\R^d$. We now extend the notion of mutual general position to finitely many polytopes $P_1,...,P_m$. For $\emptyset\not= J\subset \{1,...,m\}$ with $J=\{j_1,...,j_k\}$, let $P_J= P_{j_1}\cap\cdots\cap P_{j_k}$. We say that $P_1,...,P_m$ are in {\it mutual general position}, if $P_i$ and $P_J$ are in mutual general position for all $i\in\{1,...,m\}, \emptyset\not= J\subset\{1,...,m\}$ such that $i\notin J$. Let $U_{GP}({\cal P})$ be the class of finite unions of polytopes which are in mutual general position. We call each representation $P=\bigcup_{i=1}^m P_i$ of $P\in U_{GP}({\cal P})$, with polytopes $P_1,...,P_m$ which are  in mutual general position,  a {\it standard representation}. A standard representation is not unique, in general. For example, given a standard representation $P=\bigcup_{i=1}^m P_i$ with full dimensional polytopes $P_i$, we can add a polytope $P_{m+1}$ which lies in the interior of $P$ and is in mutual general position with all the $P_i$'s. This is e.g. the case, if $P_{m+1}$ is in the interior of $P_m$ and does not meet the boundaries of $P_1,...,P_{m-1}$. This situation is avoided, if we consider reduced representations. We call a standard representation  $ P=\bigcup_{i=1}^m P_i$ {\it reduced}, if there is no proper subfamily of $\{P_1,...,P_m\}$ which yields $P$ as the union set. It is obvious that each standard representation of $P\in U_{GP}({\cal P})$ contains (as a subcollection) a reduced standard representation. If $P=\bigcup_{i=1}^m P_i$ is a reduced standard representation, then $P_i\cap \bd P\not=\emptyset$, for $i=1,...,m$. More precisely, we then even have $P_i\setminus \bigcup_{j\not= i} P_j\not=\emptyset$. 

In the following theorem, we extend local functionals $\varphi$ and their local extensions $\Phi$ to $U_{GP}({\cal P})$ by the inclusion-exclusion formula. Since for the $d$-homogeneous part $\varphi^{(d)} = c_dV_d$ this extension is obvious, we may concentrate on the case where $c_d=0$. Then $\Phi(P,\cdot)$ is concentrated on the boundary of $P\in{\cal P}$. We therefore call such a $\varphi$ a {\it boundary functional}.

\begin{theorem} Let $\varphi$ be a boundary functional on ${\cal P}$ and let $\Phi$ be a local extension. Then $\varphi$ and $\Phi$ have  extensions to $U_{GP}({\cal P})$ which are given by the inclusion-exclusion formulas
\begin{align}\label{i-e-formula1}\varphi (P) &= \varphi (\bigcup_{i=1}^m P_i)\cr
&= \sum_{k=1}^m (-1)^{k-1} \sum_{1\le i_1<...<i_k\le m} \varphi (P_{i_1}\cap\dots\cap P_{i_k})
\end{align}
and
\begin{align}\label{i-e-formula2}\Phi (P,A) &= \Phi (\bigcup_{i=1}^m P_i,A)\cr
&= \sum_{k=1}^m (-1)^{k-1} \sum_{1\le i_1<...<i_k\le m} \Phi (P_{i_1}\cap\dots\cap P_{i_k},A ) ,
\end{align}
for any Borel set $A\subset \R^d$, where $P=\bigcup_{i=1}^m P_i $ is a standard representation. The extensions are independent of the choice of this standard representation, we have $\varphi (P) = \Phi (P,\R^d)$ and $\Phi (P,\cdot)$ is concentrated on the boundary of $P$.
\end{theorem}

\begin{proof} For the assertions, it is sufficient to discuss the local extension $\Phi$. 

For $P\in U_{GP}({\cal P})$, let $P=\bigcup_{i=1}^m P_i $ be a reduced standard representation. We define $\Phi (P,\cdot )$ by \eqref{i-e-formula2} and claim that the so-defined measure does not depend on the choice of the reduced standard representation. 
For $m=1$, nothing is to show. Let $m\ge 2$ and $P=\bigcup_{i=1}^m P_i$. If  $P=\bigcup_{j=1}^n Q_j$ is another reduced standard representation, then $n\ge 2$ and we have to show that
\begin{align}\label{i-e-formula3}\sum_{k=1}^m (-1)^{k-1}& \sum_{1\le i_1<...<i_k\le m} \Phi (P_{i_1}\cap\dots\cap P_{i_k},A ) \cr
&= \sum_{l=1}^n (-1)^{l-1} \sum_{1\le j_1<...<j_l\le n} \Phi (Q_{j_1}\cap\dots\cap Q_{j_l},A )
\end{align}
holds for all Borel sets $A\subset\R^d$.

If $A\cap P=\emptyset$, then both sides of \eqref{i-e-formula3} are zero. Let $A\subset \inte P$. We show that also then both sides of \eqref{i-e-formula3} are zero. For the left side, this is clear, if, for each $i=1,...,m$, we have $A\cap \bd P_i=\emptyset$, since then all $\Phi$-values occurring on the left side vanish. Otherwise, we can dissect $A$ appropriately and assume that $A\subset \bd P_i, i=1,...,r$, $A\subset \inte P_{i}, i=r+1,...,s$, and $A\cap P_i =\emptyset, i=s+1,...,m$, for some $1<r<s\le m$ (and after a suitable re-enumeration).  Then \eqref{i-e-formula2} reduces to 
\begin{align*}\Phi (P,A) &= \sum_{k=1}^{s} (-1)^{k-1} \sum_{1\le i_1<...<i_k\le s} \Phi (P_{i_1}\cap\dots\cap P_{i_k},A )\cr
&= \sum_{k=1}^{s-1} (-1)^{k-1} \sum_{1\le i_1<...<i_k\le s-1} \Phi (P_{i_1}\cap\dots\cap P_{i_k},A )\cr
&\ -\sum_{k=1}^{s-1} (-1)^{k-1} \sum_{1\le i_1<...<i_k\le s-1} \Phi (P_{i_1}\cap\dots\cap P_{i_k}\cap P_{s},A )\cr
&= 0,
\end{align*}
since $\Phi (P_{i_1}\cap\dots\cap P_{i_k},\cdot)$, for $1\le i_1<...<i_k\le s-1$, is locally defined and therefore
$$\Phi (P_{i_1}\cap\dots\cap P_{i_k},A) = \Phi (P_{i_1}\cap\dots\cap P_{i_k}\cap P_{s},A) .$$ 
By similar arguments, the right side of \eqref{i-e-formula3} is shown to be zero.

Finally, let $A\subset \bd P$. By dissecting $A$ appropriately into finitely many pieces, we may assume that $A$ does not intersect more than one disjoint closed part of $P_i\cap \bd P$, for each $i=1,...,m$, and not more than one disjoint closed part of $Q_j\cap \bd P$, for each $j=1,...,n$. We can find an open neighborhood $U$ of $A$ such that $P_i\cap A\not=\emptyset$, if and only if $(P_i\setminus \bigcup_{s\not= i}P_s)\cap (U\cap \bd P)\not=\emptyset$, and similarly $Q_j\cap A\not=\emptyset$, if and only if $(Q_j\setminus \bigcup_{t\not= j}Q_t)\cap (U\cap \bd P)\not=\emptyset$. We may assume, after a suitable re-enumeration, that $P_1\cap A\not=\emptyset,...,P_r\cap A\not=\emptyset$ and $P_{r+1}\cap A= \emptyset,...,P_m\cap A=\emptyset$, for some $1\le r\le m$. The fact that then 
$(P_i\setminus \bigcup_{s\not= i}P_s)\cap (U\cap \bd P)\not=\emptyset$ (for $i\in\{ 1,...,r\}$) guarantees that there is a $j=j(i)$ such that 
$$(P_i\setminus \bigcup_{s\not= i}P_s)\cap (U\cap \bd P) = (Q_{j(i)}\setminus \bigcup_{t\not= j(i)}Q_t)\cap (U\cap \bd P).$$ 
Interchanging the roles of the $P_i$ and the $Q_j$, and then re-enumerating the $Q_j$, we obtain that $P_1\cap U=Q_1\cap U,...,P_r\cap U=Q_r\cap U$ (which also implies $r\le n$) and $P_i\cap U = Q_j\cap U=\emptyset$ for $i,j>r$.

 Obviously, this yields
$$
P_{i_1}\cap\dots\cap P_{i_k} \cap A= Q_{i_1}\cap\dots\cap Q_{i_k} \cap A ,
$$
for $1\le i_1<...<i_k\le r, k=1,...,r$, and 
$$
P_{i_1}\cap\dots\cap P_{i_k} \cap A= \emptyset, \quad Q_{j_1}\cap\dots\cap Q_{j_l} \cap A =\emptyset,
$$
if at least on  index $i_s > r$ (respectively $j_t>r$) appears. This shows that \eqref{i-e-formula3} holds. 

We have shown that a definition of $\Phi(P,\cdot)$ by \eqref{i-e-formula2} is independent of the representation $P=\bigcup_{i=1}^m P_i$, provided this is a reduced standard representation.  If we add another polytope $P_{m+1}$ to a standard representation $P=\bigcup_{i=1}^m P_i $ such that $P=\bigcup_{i=1}^{m+1} P_i$ is a standard representation, then $P_{m+1}\cap \bd P =\emptyset$ and therefore $P_{m+1}\cap V =\emptyset$ for an open neighborhood $V$ of $\bd P$. Then, the above arguments go through, we obtain \eqref{i-e-formula2} for each standard representation of $P$ and we also have seen that $\Phi (P,\cdot)$ is concentrated on $\bd P$. 
\end{proof}

From this result, it follows that every local functional $\varphi$ on ${\cal P}$ (as well as its local extension $\Phi$) has an extension to $U_{GP}({\cal P})$ given by the inclusion-exclusion formula. We call these extensions of $\varphi$ and $\Phi$ the {\it $GP$-extensions}. 

\section{Poisson processes}

In this section, we consider a Poisson particle process $X$ with convex grains, that has a translation regular and locally finite intensity measure $\Theta\not\equiv 0$ (see \cite[Section 11.1]{SW}, for details). Thus, 
\begin{equation}\label{16.1.1}
\Theta (A) = \int_{{\cal K}_0} \int_{{\mathbb R}^d} {\bf 1}_A(K+x) \eta (K,x)\,\lambda (dx)\, {\mathbb Q} (d K),\qquad A\in{\cal B}({\cal K}),
\end{equation}
where ${\cal K}_0$ denotes the set of convex bodies $K$ with circumcenter $c(K)$ at the origin and ${\cal B}({\cal K})$ is the $\sigma$-algebra of Borel sets in $\cal K$. If $\eta$ does not depend on $K$, the {\it spatial intensity function} $\eta$ and the {\it grain distribution} $\Q$ are uniquely determined by \eqref{16.1.1}. The Poisson process $X$ is {\it stationary}, if and only if $\eta$ is a constant $\gamma >0$ (the {\it intensity}). 

We assume, throughout the following, that ${\mathbb Q}$ is concentrated on ${\cal P}_0 := {\cal K}_0 \cap {\cal P}$, hence the particles are (almost surely) convex polytopes. The local finiteness of $\Theta$ is then equivalent to 
\begin{equation}\label{integrability1}
\int_{{\cal P}_0} \int_{{\mathbb R}^d} {\bf 1}\{(P+x)\cap C\not=\emptyset\} \eta (P,x)\,\lambda (dx)\, {\mathbb Q} (d P)<\infty,
\end{equation}
for any compact $C\subset\R^d$ (see \cite[(11.4)]{SW}). 

Let $\varphi$ be a local functional on $\cal P$  with local extension $\Phi$ and let $\Phi^{(j)}_{m_1,...,m_k}$ be the corresponding mixed kernels which exist by Theorem \ref{trans1}. Without loss of generality, we assume $\Phi\ge 0$ and hence  $\Phi^{(j)}_{m_1,...,m_k}\ge 0$, for all $j\in\{ 0,\dots ,d\}$, $k\ge 2$, and $m_1,\dots,m_k\in \{j,\dots,d\}$ with 
$\sum_{i=1}^k m_i = (k-1)d+j.$ In order to unify the presentation, we also assume now $c_d=1$ and put  $\Phi^{(j)}_{j}= \Phi^{(j)}, j=0,...,d$. 

The following result is an analog of Corollary 11.1.4 in \cite{SW} and follows in a similar way from Theorem  \ref{trans1} above. For $k\in {\mathbb N}$, the process $X^k_{\not=}$ consists of $k$-tuples $(P_1,...,P_k)$ of (pairwise) different polytopes $P_i\in X$. The intensity measure of $X^k_{\not=}$ is the {\it $k$th factorial moment measure} of $X$.

 \begin{theorem}\label{Poisson} Let $X$ be a Poisson process of convex polytopes in ${\mathbb R}^d$ with translation regular and locally finite intensity measure, let $k\in{\mathbb N}$, $j\in\{ 0,\dots ,d\}$ and $m_1,\dots,m_k\in \{j,\dots,d\}$ with 
$$
\sum_{i=1}^k m_i = (k-1)d+j.
$$
Assume that
\begin{align}\label{integrability2}
&\int_{({\cal P}_0)^k}\int_{({\mathbb R}^d)^k} \varphi^{(j)}_{m_1,\dots ,m_k}(P_1,\dots ,P_k){\bf 1}\{P_1^{x_1}\cap C\not=\emptyset\} \eta (P_1,x_1)\cdots \cr
&\times {\bf 1}\{P_k^{x_k}\cap C\not=\emptyset\} \eta (P_k,x_k)\,\lambda^k (d(x_1,...,x_k))\,{\mathbb Q}^k (d (P_1,...,P_k))<\infty ,\quad\quad\end{align}
for any compact $C\subset\R^d$.

Then, 
$${\mathbb E} \sum_{(P_1,\dots ,P_k)\in X^k_{\not=}} \Phi^{(j)}_{m_1,\dots ,m_k}(P_1,\dots ,P_k;\cdot )$$ 
is a locally finite measure on $({\mathbb R}^d)^k$ which is absolutely continuous with respect to $\lambda^k$, and a density is given by  
\begin{eqnarray*}
&&\overline \varphi_{ m_1,\dots,m_k}^{(j)}(X,\dots ,X;z_1,\dots ,z_k) \\ 
&&= \int_{({\cal P}_0)^k}\int_{({\mathbb R}^d)^k} \eta(P_1,z_1-x_1)\cdots \eta(P_k,z_k-x_k)\,\\
&& \hspace*{4mm}\times\,\Phi^{(j)}_{m_1,\dots,m_k}(P_{1},\dots,P_{k};d (x_1,\dots ,x_k))\,{\mathbb Q}^k (d (P_1,..., P_k))
\end{eqnarray*}
for $\lambda^k$-almost all $(z_1,\dots,z_k)\in({\mathbb R}^d)^k$.
\end{theorem}

\begin{proof} We give a sketch of the proof and refer to \cite{SW}, for details and definitions. First, by Campbell's theorem, the form of factorial moment measures for Poisson processes, and Fubini's theorem, we obtain, for bounded Borel sets $B_1,...,B_k$ in $\R^d$,
\begin{align*}
&{\mathbb E} \sum_{(P_1,\dots ,P_k)\in X^k_{\not=}} \Phi^{(j)}_{m_1,\dots ,m_k}(P_1,\dots ,P_k;B_1\times\cdots\times B_k )\cr
&\  = \int_{({\cal P}_0)^k}\int_{({\mathbb R}^d)^k} \Phi^{(j)}_{m_1,\dots,m_k}(P_{1}^{y_1},\dots,P_{k}^{y_k}; B_1\times\cdots \times B_k)\cr
&\quad  \times \eta(P_1,y_1)\cdots \eta(P_k,y_k)\lambda^k(d(y_1,...,y_k)) {\mathbb Q}^k (d (P_1,...,P_k)).
\end{align*}
We choose $r>0$, such that $B_1,...,B_k$ are contained in the interior of the cube $rW$. Then
\begin{align*}
&\Phi^{(j)}_{m_1,\dots,m_k}(P_{1}^{y_1},\dots,P_{k}^{y_k}; B_1\times\cdots \times B_k)\cr
&\ \le \varphi^{(j)}_{m_1,\dots,m_k}(P_{1},\dots,P_{k})
{\bf 1}\{P_{1}^{y_1}\cap rW\not=\emptyset\}\cdots {\bf 1}\{P_{k}^{y_k}\cap rW\not=\emptyset\} .
\end{align*}
From \eqref{integrability2}, we thus obtain
\begin{align*}
&{\mathbb E} \sum_{(P_1,\dots ,P_k)\in X^k_{\not=}} \Phi^{(j)}_{m_1,\dots ,m_k}(P_1,\dots ,P_k;B_1\times\cdots\times B_k )\cr
& \le \int_{({\cal P}_0)^k}\int_{({\mathbb R}^d)^k} \varphi^{(j)}_{m_1,\dots ,m_k}(P_1,\dots ,P_k){\bf 1}\{P_1^{y_1}\cap rW\not=\emptyset\} \eta (P_1,y_1)\cdots \cr
&\quad\times {\bf 1}\{P_k^{y_k}\cap rW\not=\emptyset\} \eta (P_k,y_k)\,\lambda^k (d(y_1,...,y_k))\,{\mathbb Q}^k (d (P_1,...,P_k))\cr
&<\infty .
\end{align*}
Therefore, 
$${\mathbb E} \sum_{(P_1,\dots ,P_k)\in X^k_{\not=}} \Phi^{(j)}_{m_1,\dots ,m_k}(P_1,\dots ,P_k;\cdot )$$ 
is locally finite.

In the same way, we get
\begin{align*}
&{\mathbb E} \sum_{(P_1,\dots ,P_k)\in X^k_{\not=}} \Phi^{(j)}_{m_1,\dots ,m_k}(P_1,\dots ,P_k;B_1\times\cdots\times B_k )\cr
&\  = \int_{{\cal P}_0}\dots\int_{{\cal P}_0}\int_{({\mathbb R}^d)^k} \Phi^{(j)}_{m_1,\dots,m_k}(P_{1}^{y_1},\dots,P_{k}^{y_k}; B_1\times\cdots \times B_k)\cr
&\quad  \times \eta(P_1,y_1)\cdots \eta(P_k,y_k)\lambda^k(d(y_1,...,y_k)) {\mathbb Q} (d P_1)\cdots \,{\mathbb Q} (d P_k)\cr
&\  = \int_{{\cal P}_0}\dots\int_{{\cal P}_0}\int_{({\mathbb R}^d)^k} \int_{({\mathbb R}^d)^k}
{\bf 1}_{B_1-y_1}(x_1)\cdots{\bf 1}_{B_k-y_k}(x_k) 
 \eta(P_1,y_1)\cdots \eta(P_k,y_k)\cr
&\quad \times\Phi^{(j)}_{m_1,\dots,m_k}(P_{1} ,\dots,P_{k} ; d(x_1,...,x_k))\lambda^k(d(y_1,...,y_k)) {\mathbb Q} (d P_1)\cdots \,{\mathbb Q} (d P_k)\cr
&\  = \int_{{\cal P}_0}\dots\int_{{\cal P}_0}\int_{({\mathbb R}^d)^k} \int_{({\mathbb R}^d)^k}
{\bf 1}_{B_1}(z_1)\cdots{\bf 1}_{B_k}(z_k) 
 \eta(P_1,z_1-x_1)\cdots \eta(P_k,z_k-x_k)\cr
 &\quad \times\Phi^{(j)}_{m_1,\dots,m_k}(P_{1} ,\dots,P_{k} ; d(x_1,...,x_k))\lambda^k(d(z_1,...,z_k)) {\mathbb Q} (d P_1)\cdots \,{\mathbb Q} (d P_k)\cr
&\  = \int_{B_1\times\cdots\times B_k}\Biggl(\int_{{\cal P}_0}\dots\int_{{\cal P}_0}\int_{({\mathbb R}^d)^k} 
 \eta(P_1,z_1-x_1)\cdots \eta(P_k,z_k-x_k)\cr
 &\quad \times\Phi^{(j)}_{m_1,\dots,m_k}(P_{1} ,\dots,P_{k} ; d(x_1,...,x_k)){\mathbb Q} (d P_1)\cdots \,{\mathbb Q} (d P_k)\Biggr)\lambda^k(d(z_1,...,z_k)) 
\end{align*}
which shows the absolute continuity and the stated form of the density.
\end{proof}

For $k=1$, we have $\Phi^{(j)}_j = \Phi^{(j)}$ and correspondingly write $\overline\varphi^{(j)}(X,\cdot)$ for $\overline\varphi^{(j)}_j(X;\cdot)$. The mean value $\overline\varphi^{(j)}(X,\cdot)$ thus gives the specific $\varphi^{(j)}$-value for $X$ (the mean value of $\varphi^{(j)}$ per unit volume of $X$). The integrability assumption \eqref{integrability2} then reduces to
$$
\int_{{\cal P}_0}\int_{{\mathbb R}^d} \varphi^{(j)}(P){\bf 1}\{P^{x}\cap C\not=\emptyset\} \eta (P,x)\,\lambda (dx)\,{\mathbb Q}
(dP)<\infty ,$$
for compact $C\subset\R^d$. 

For $k>d$, the decomposition property \eqref{decomp} implies that at least $k-d$ of the indices $m_1,...,m_k$ are equal to $d$. If we assume, without loss of generality, that $m_{d+1}=\cdots =m_k=d$, then the integrability condition \eqref{integrability2} can be replaced by a simpler one and the assertion of Theorem \ref{Poisson} can also be simplified. Namely,  it is then sufficient to require 
\begin{align}\label{integrability3}
&\int_{({\cal P}_0)^d}\int_{({\mathbb R}^d)^d} \varphi^{(j)}_{m_1,\dots ,m_d}(P_1,\dots ,P_d){\bf 1}\{P_1^{x_1}\cap C\not=\emptyset\} \eta (P_1,x_1)\cdots \cr
&\times {\bf 1}\{P_d^{x_d}\cap C\not=\emptyset\} \eta (P_d,x_d)\,\lambda^d (d(x_1,...,x_d))\,{\mathbb Q}^d (d (P_1,...,P_d))<\infty .
\end{align}
In the proof, it is then used that
\begin{align*}
&{\mathbb E} \sum_{(P_1,\dots ,P_k)\in X^k_{\not=}} \Phi^{(j)}_{m_1,\dots ,m_k}(P_1,\dots ,P_k;B_1\times\cdots\times B_k )\cr
&\  = \int_{({\cal P}_0)^k}\int_{({\mathbb R}^d)^k} \Phi^{(j)}_{m_1,\dots,m_k}(P_{1}^{y_1},\dots,P_{k}^{y_k}; B_1\times\cdots \times B_k)\cr
&\quad  \times \eta(P_1,y_1)\cdots \eta(P_k,y_k)\lambda^k(d(y_1,...,y_k)) {\mathbb Q}^k (d (P_1,...,P_k))\cr
&\  = \int_{({\cal P}_0)^d}\int_{({\mathbb R}^d)^d} \Phi^{(j)}_{m_1,\dots,m_d}(P_{1}^{y_1},\dots,P_{d}^{y_d}; B_1\times\cdots \times B_d)\cr
&\quad  \times \eta(P_1,y_1)\cdots \eta(P_d,y_d)\lambda^d(d(y_1,...,y_d)) {\mathbb Q}^d (d (P_1,...,P_d))\cr
&\quad\times \prod_{i=d+1}^k \int_{{\cal P}_0}\int_{{\mathbb R}^d} \lambda_{P^y}( B_i) \eta(P,y)\lambda(dy) {\mathbb Q}
(dP)\cr
&\ \le \int_{({\cal P}_0)^d}\int_{({\mathbb R}^d)^d} \varphi^{(j)}_{m_1,\dots ,m_d}(P_1,\dots ,P_d){\bf 1}\{P_1^{y_1}\cap rW\not=\emptyset\} \eta (P_1,y_1)\cdots \cr
&\quad\times {\bf 1}\{P_d^{y_d}\cap rW\not=\emptyset\} \eta (P_d,y_d)\,\lambda^d (d(y_1,...,y_d))\,{\mathbb Q}^d (d (P_1,...,P_k))\cr
&\quad\times\left( \int_{{\cal P}_0}\int_{{\mathbb R}^d} {\bf 1}\{P^{y}\cap rW\not=\emptyset\}  \eta(P,y)\lambda(dy) {\mathbb Q}
(dP)\right)^{k-d}<\infty ,
\end{align*}
in view of \eqref{integrability3} and \eqref{integrability1}.

For the density, we then get
\begin{align*}
&\overline \varphi_{ m_1,\dots,m_d,d,...,d}^{(j)}(X,\dots ,X;z_1,\dots ,z_k) \cr
&= \int_{({\cal P}_0)^k}\int_{({\mathbb R}^d)^k} \eta(P_1,z_1-x_1)\cdots \eta(P_k,z_k-x_k)\cr
& \quad\times\,\Phi^{(j)}_{m_1,\dots,m_d,d,...,d}(P_{1},\dots,P_{k};d (x_1,\dots ,x_k))\,{\mathbb Q}^k (d (P_1,..., P_k))\cr
&= \int_{({\cal P}_0)^d}\int_{({\mathbb R}^d)^d} \eta(P_1,z_1-x_1)\cdots \eta(P_d,z_d-x_d)\cr
& \quad\times\,\Phi^{(j)}_{m_1,\dots,m_d}(P_{1},\dots,P_{d};d (x_1,\dots ,x_d))\,{\mathbb Q}^d (d (P_1,..., P_d))\cr
& \quad\times \prod_{i=d+1}^k \int_{{\cal P}_0}\int_{P}\eta(P,z_i-x)\lambda (dx){\mathbb Q}(dP)\cr
&= \overline \varphi_{ m_1,\dots,m_d}^{(j)}(X,\dots ,X;z_1,\dots ,z_d)\overline V_d(X,z_{d+1})\cdots \overline V_d(X,z_{k}) .
\end{align*}

We next introduce and study intersection densities. For $k\ge 2$, we call 
$$X_k :=\{ P_1\cap\dots\cap P_k : (P_1,...,P_k)\in X^k_{\not=}\}$$ the $k$th {\it intersection process} of $X$. It is again a process of polytopes, but not a Poisson process anymore. However, it is almost surely simple (each polytope appears at most once). Its intensity measure $\Theta_k$ is the image of the measure
\begin{align*}
A\mapsto & \frac{1}{k!}\int_{({\cal P}_0)^k}\int_{({\mathbb R}^d)^k} {\bf 1}\{(P_1^{y_1},...,P_k^{y_k})\in A\} \eta (P_1,y_1)\cdots \eta (P_k,y_k)\cr
&\quad\times \lambda^k (d(y_1,...,y_k))\,{\mathbb Q}^k (d (P_1,...,P_k)),\quad A\in{\cal B} (({\cal P})^k),
\end{align*}
under the measurable mapping $(P_1,...,P_k)\mapsto P_1\cap\dots\cap P_k$. We now define the $k$th {\it intersection density} of $X$ for the local functional $\varphi$ (with local extension $\Phi$) as the Radon-Nikodym derivative $\overline\varphi (X_k,\cdot)$ of the measure
$${\mathbb E} \sum_{P\in X_k} \Phi (P,\cdot)
$$
with respect to $\lambda$ (provided this measure is absolutely continuous). In the following theorem, we concentrate on the $j$-homogeneous part $\varphi^{(j)}$ of $\varphi$, show that the intersection density $\overline\varphi^{(j)} (X_k,\cdot)$ exists and give an explicit formula for it. The result follows from Theorem \ref{Poisson} together with Theorem \ref{trans1}. The proof is analogous to the one given above, therefore we leave it out. 
 
\begin{theorem}\label{Poisson2} Let $X$ be a Poisson process of convex polytopes in ${\mathbb R}^d$ with translation regular and locally finite intensity measure, let $j\in\{ 0,\dots ,d\}$, and let $k\ge 2$. 
Assume that \eqref{integrability2} holds
for compact $C\subset\R^d$ and all $m_1,\dots,m_k\in \{j,\dots,d\}$ with 
$\sum_{i=1}^k m_i = (k-1)d+j.$

Then, 
$${\mathbb E} \sum_{P\in X_k} \Phi^{(j)}(P,\cdot )$$ 
is a locally finite measure on ${\mathbb R}^d$ which is absolutely continuous with respect to $\lambda$, and a density is given by  
\begin{align*}
\overline \varphi^{(j)}(X_k,z)&= \sum_{m_1,\dots, m_k=j\atop m_1+\dots +m_k=(k-1)d+j}^{d} \frac{1}{k!}\int_{({\cal P}_0)^k}\int_{({\mathbb R}^d)^k} \eta(P_1,z-x_1)\cdots \eta(P_k,z-x_k)\cr
& \hspace*{4mm}\times\,\Phi^{(j)}_{m_1,\dots,m_k}(P_{1},\dots,P_{k};d (x_1,\dots ,x_k))\,{\mathbb Q}^k (d (P_1,..., P_k)) ,
\end{align*}
for $\lambda^k$-almost all $z\in{\mathbb R}^d$.
\end{theorem}

Comparing the result with Theorem \ref{Poisson}, we obtain
$$
\overline \varphi^{(j)}(X_k,z)= \sum_{m_1,\dots, m_k=j\atop m_1+\dots +m_k=(k-1)d+j}^{d} \frac{1}{k!}\
\overline \varphi_{ m_1,\dots,m_k}^{(j)}(X,\dots ,X;z,\dots ,z) .
$$
We can use the decomposition property \eqref{decomp} to simplify some of the formulas. Namely, for $j=d$, we get
\begin{align*}
\overline \varphi^{(d)}(X_k,z)&= \frac{1}{k!}\overline V_d(X,z)^{k}  
\end{align*}
(remember that we assumed $c_d=1$, in this section). For $j<d$, 
let $s$ be the number of indices $m_1,...,m_k$ which are smaller than $d$, $1\le s\le (d-j)\wedge k$ (here $\wedge$ denotes the minimum). Then,
\begin{align*}
&\sum_{m_1,\dots, m_k=j\atop m_1+\dots +m_k=(k-1)d+j}^{d} \frac{1}{k!}\
\overline \varphi_{ m_1,\dots,m_k}^{(j)}(X,\dots ,X;z,\dots ,z)\cr
&\quad =\sum_{s=1}^{(d-j)\wedge k}{k\choose s} \sum_{m_1,\dots, m_s=j\atop m_1+\dots +m_d=(s-1)d+j}^{d-1} \frac{1}{k!}\
\overline \varphi_{ m_1,\dots,m_s}^{(j)}(X,\dots ,X;z,\dots ,z) \overline V_d(X,z)^{k-s}\cr
&\quad = \sum_{s=1}^{(d-j)\wedge k} \sum_{m_1,\dots, m_s=j\atop m_1+\dots +m_d=(s-1)d+j}^{d-1} \frac{1}{s!}\
\overline \varphi_{ m_1,\dots,m_s}^{(j)}(X,\dots ,X;z,\dots ,z)\frac{\overline V_d(X,z)^{k-s}}{(k-s)!} . 
\end{align*}

\begin{koro}\label{koro1} 
For the intersection densities, we have
\begin{align*}
\overline \varphi^{(d)}(X_k,z)&= \frac{1}{k!}\overline V_d(X,z)^{k}  
\end{align*}
and
\begin{align*}
&\overline \varphi^{(j)}(X_k,z)\cr
&\ = \sum_{s=1}^{(d-j)\wedge k} \sum_{m_1,\dots, m_s=j\atop m_1+\dots +m_d=(s-1)d+j}^{d-1} \frac{1}{s!}\
\overline \varphi_{ m_1,\dots,m_s}^{(j)}(X,\dots ,X;z,\dots ,z)\frac{\overline V_d(X,z)^{k-s}}{(k-s)!}  
\end{align*}
for $j\in\{ 0,...,d-1\}$ and $k=1,2,...$.
\end{koro}

\section{Boolean models}

We now consider, for a Poisson process $X$ as it was discussed in the previous section, the union set $$
Z=\bigcup_{P\in X} P.$$
$Z$ is a random closed set in $\R^d$, a {\it Boolean model} with convex grains, which in our case are polytopes.
Combining Theorem \ref{Poisson2} with \cite[Theorem 11.1.2]{SW}, we obtain the following extension of \cite[Theorem 11.1.3]{SW}. Here, for a local functional $\varphi$ with local extension $\Phi$, we make use of the fact that $\Phi$ extends to sets in the extended convex ring which are locally $GP$-unions as a signed Radon measure. As we shall see, under the assumptions of this section, the union set $Z$ is almost surely a $GP$-union. Therefore, $\Phi(Z,\cdot )$ is a random signed Radon measure and its expectation ${\mathbb E}\Phi(Z,\cdot )$ is a signed Radon measure. For the following result, we need the integrability condition \eqref{integrability2}, for $j\in\{ 0,\dots ,d\}$ and all $k\in{\mathbb N}$ and $m_1,\dots,m_k\in \{j,\dots,d\}$ with $
\sum_{i=1}^k m_i = (k-1)d+j.
$ Notice that, in view of the decomposability property \eqref{decomp}, this is actually only a requirement for finitely many values of $k$.  

\begin{theorem}\label{T16.1.4}
Let $Z$ be a Boolean model in ${\mathbb R}^d$ with polytopal grains and let $\varphi$ be a local functional on ${\cal P}$ with local extension $\Phi$. Assume that \eqref{integrability2} holds, for $j\in\{ 0,\dots ,d\}$ and all $k\in{\mathbb N}$ and $m_1,\dots,m_k\in \{j,\dots,d\}$ with $\sum_{i=1}^k m_i = (k-1)d+j$. 

Then, for $j=0,...,d$, the expectation ${\mathbb E}\Phi^{(j)}(Z,\cdot )$ is a signed Radon measure which is absolutely continuous with respect to $\lambda$. For $\lambda$-almost all $z$, its density $\overline \varphi^{(j)}(Z,\cdot)$ satisfies
\begin{equation*}
\overline \varphi^{(d)}(Z,z) = 1 -\E^{-\overline \varphi^{(d)}(X,z)},
\end{equation*} 
$$
\overline \varphi^{(d-1)}(Z,z) = \E^{-\overline \varphi^{(d)}(X,z)} \overline \varphi^{(d-1)}(X,z),
$$
and
\begin{eqnarray*}
\overline  \varphi^{(j)}(Z,z)&=&\E^{-\overline \varphi^{(d)}(X,z)}\left(\overline \varphi^{(j)}(X,z) -\sum_{s=2}^{d-j}\frac{(-1)^{s}}{s!}\right.\\
&&\hspace*{4mm}\times\,\sum_{m_1,\dots, m_s=j+1\atop m_1+\dots +m_s=(s-1)d+j}^{d-1} \overline \varphi^{(j)}_{m_1,\dots ,m_s}(X,\dots,X;z,\dots ,z)\Biggr),
\end{eqnarray*}
for $j=0,\dots ,d-2.$
\end{theorem}

Here, the densities for $X$ are defined as in Theorem \ref{Poisson}.

\begin{proof} We first show that, for given $r>0$, the truncated set $$Z^{(r)} = \bigcup_{P\in X, P\cap rW\not=\emptyset} P$$
is almost surely a $GP$-union. Namely, assume that the number of polytopes $P\in X$ with $P\cap rW\not=\emptyset$ is $m>0$, then the conditional distribution of these $m$ random polytopes is (up to a normalizing factor) the image of the measure
$$
\int\int \eta (P_1,x_1)\cdots \eta (P_m,x_m)\,\lambda^m (d(x_1,...,x_m))\, {\mathbb Q}^m (d( P_1,...,P_m))
$$
under $(P_1,...,P_m,x_1,...,x_m)\mapsto (P_1+x_1,...,P_m+x_m)$. It is therefore sufficient to show that, for any given polytopes $P,Q$ and independent random translations $\xi, \phi$ distributed according to the (normalized) measure $$A\mapsto \int_{A\cap (rW-P)} \eta (P,x)\lambda (dx)$$ respectively $$A\mapsto \int_{A\cap (rW-P)} \eta (Q,x)\lambda (dx),$$ the polytopes $P+\xi$ and $Q+\phi$ are in mutual general position. This, however, follows from the fact that, for any $P,Q\in{\cal P}$, the set of translations $x$, such that $P$ and $Q+x$ are not in mutual general position, has Lebesgue measure zero.

Now let $B\in{\cal B}$ be a bounded Borel set and choose $r>0$ such that $B$ is contained in the interior of $rW$. Since $Z\cap rW= Z^{(r)}\cap rW$ and since the local extension $\Phi$ is locally defined, we can put $\Phi (Z,B) = \Phi (Z^{(r)},B)$ where the latter value is obtained by $GP$-extension and is almost surely defined. In this way, we get a signed Radon measure $\Phi(Z,\cdot)$ for almost all realizations of $Z$. The same holds for the $j$-homogeneous parts $\Phi^{(j)}(Z,\cdot), j=0,...,d$. If $\nu$ is the (random) number of particles $Q_1,...,Q_\nu$ in $X$ which meet $rW$, we can use the inclusion-exclusion formula to obtain
\begin{align}\label{i-e-f}
\Phi^{(j)} (Z,B) &= \Phi^{(j)} (\bigcup_{i=1}^\nu Q_i,B)\cr
&= \sum_{k=1}^\nu (-1)^{k-1} \sum_{1\le i_1<...<i_k\le\nu} \Phi^{(j)} (Q_{i_1}\cap\dots\cap Q_{i_k}, B)\cr
&= \sum_{k=1}^\infty \frac{(-1)^{k-1}}{k!} \sum_{(P_1,...,P_k)\in X^k_{\not =}} \Phi^{(j)} (rW\cap P_1\cap\dots\cap P_k, B) .
\end{align}
Hence
\begin{align*}
{\mathbb E}|\Phi^{(j)} (Z,B)| &\le \sum_{k=1}^\infty \frac{1}{k!}\ {\mathbb E} \sum_{(P_1,...,P_k)\in X^k_{\not =}} \Phi^{(j)} (rW\cap P_1\cap\dots\cap P_k, B).
\end{align*}
As in the proof of Theorem \ref{Poisson}, we get
\begin{align*}
{\mathbb E}& \sum_{(P_1,...,P_k)\in X^k_{\not =}} \Phi^{(j)} (rW\cap P_1\cap\dots\cap P_k, B)\cr
&=\int_{({\cal P}_0)^k}\int_{({\mathbb R}^d)^k} \Phi^{(j)}(rW\cap P_{1}^{y_1}\cap\dots\cap P_{k}^{y_k}, B)\cr
&\quad  \times \eta(P_1,y_1)\cdots \eta(P_k,y_k)\lambda^k(d(y_1,...,y_k)) {\mathbb Q}^k (d (P_1,...,P_k))\cr
& = \sum_{{m_0,\dots,m_k=j}\atop{m_0+\dots +m_k=kd+j}}^d \int_{({\cal P}_0)^k}\int_{({\mathbb R}^d)^{k+1}}{\bf 1}_B(x_0)  \eta(P_1,x_0-x_1)\cdots \eta(P_k,x_0-x_k)\cr
&\quad \times \Phi^{(j)}_{m_0,m_1,\dots,m_k}(rW,P_1,\dots,P_k;d(x_0,x_1,...,x_k)){\mathbb Q}^k (d (P_1,...,P_k))\cr
& = \sum_{{m_1,\dots,m_k=j}\atop{m_1+\dots +m_k=(k-1)d+j}}^d \int_{B}\int_{({\cal P}_0)^k}\int_{({\mathbb R}^d)^{k}} \eta(P_1,x_0-x_1)\cdots \eta(P_k,x_0-x_k)\cr
&\quad \times \Phi^{(j)}_{m_1,\dots,m_k}(P_1,\dots,P_k;d(x_1,...,x_k)){\mathbb Q}^k (d (P_1,...,P_k))\lambda (dx_0). 
\end{align*}
Here, we used formula \eqref{mixedint2}, the decomposition property of mixed measures and the fact that, since  $B$ lies in the interior of $rW$, the summands with $m_0\not=d$ vanish.

The above summation rule shows that for $k>d$, at least $k-d$ of the parameters $m_1,...,m_k$ have to be $d$. Thus the decomposition property can be used again to show that, for $k>d$,
\begin{align*}
{\mathbb E}& \sum_{(P_1,...,P_k)\in X^k_{\not =}} \Phi^{(j)} (rW\cap P_1\cap\dots\cap P_k, B)\cr
& = \sum_{{m_1,\dots,m_d=j}\atop{m_1+\dots +m_d=(d-1)d+j}}^d {k\choose d}\int_{B}\Biggl[\int_{({\cal P}_0)^d}\int_{({\mathbb R}^d)^{d}} \eta(P_1,x_0-x_1)\cdots \eta(P_d,x_0-x_d)\cr
&\quad \times \Phi^{(j)}_{m_1,\dots,m_d}(P_1,\dots,P_d;d(x_1,...,x_d)){\mathbb Q}^d (d (P_1,...,P_d))\Biggr]\cr
&\quad \times\left[\int_{{\cal P}_0}\int_{{\mathbb R}^d} \eta(Q,x_0-y)\lambda_Q(dy){\mathbb Q}(dQ)\right]^{k-d}\lambda (dx_0)\cr
&\le \Biggl[\sum_{{m_1,\dots,m_d=j}\atop{m_1+\dots +m_d=(d-1)d+j}}^d \int_{B}\int_{({\cal P}_0)^d}\int_{({\mathbb R}^d)^{d}} \eta(P_1,x_0-x_1)\cdots \eta(P_d,x_0-x_d)\cr
&\quad \times \Phi^{(j)}_{m_1,\dots,m_d}(P_1,\dots,P_d;d(x_1,...,x_d)){\mathbb Q}^d (d (P_1,...,P_d))\lambda(dx_0)\Biggr]\cr
&\quad\times{k\choose d} \left[\int_{{\cal P}_0}\int_{{\mathbb R}^d} {\bf 1}\{ Q^x\cap B\not=\emptyset\}\eta(Q,x)\lambda(dx){\mathbb Q}(dQ)\right]^{k-d}
\end{align*} 
From Theorem \ref{Poisson}, we obtain that the first square bracket has a finite value $C_1$. Condition \eqref{integrability1} implies that the second square bracket is a finite number $C_2$. 
It also follows from Theorem \ref{Poisson} and the above calculations, that
\begin{align*}&\sum_{k=1}^d \frac{1}{k!}\sum_{{m_1,\dots,m_k=j}\atop{m_1+\dots +m_k=(k-1)d+j}}^d \int_{B}\int_{({\cal P}_0)^k}\int_{({\mathbb R}^d)^{k}} \eta(P_1,x_0-x_1)\cdots \eta(P_k,x_0-x_k)\cr
&\quad \times \Phi^{(j)}_{m_1,\dots,m_k}(P_1,\dots,P_k;d(x_1,...,x_k)){\mathbb Q}^k (d (P_1,...,P_k))\lambda (dx_0) 
\end{align*}
is a finite number $C_0$.

Altogether, we get
\begin{align*}
|{\mathbb E}\Phi^{(j)} (Z,B)| &\le C_0+C_1\sum_{k=d}^\infty \frac{1}{k!} {k\choose d} C_2^{k-d}=C_0+\frac{C_1}{d!}\E^{C_2}<\infty,
\end{align*}
which shows that ${\mathbb E}\Phi^{(j)} (Z,\cdot )$ is a signed Radon measure.

The calculation also shows that we can interchange summation and expectation in \eqref{i-e-f} to obtain (with Theorem \ref{Poisson2})
\begin{align*}
&{\mathbb E}  \Phi^{(j)} (Z,B) \cr
&\ = \sum_{k=1}^\infty \frac{(-1)^{k-1}}{k!} {\mathbb E}\sum_{(P_1,...,P_k)\in X^k_{\not =}} \Phi^{(j)} (P_1\cap\dots\cap P_k, B)\cr
&\ = \sum_{k=1}^\infty \frac{(-1)^{k-1}}{k!}\sum_{{m_1,\dots,m_k=j}\atop{m_1+\dots +m_k=(k-1)d+j}}^d \int_{B}\int_{({\cal P}_0)^k}\int_{({\mathbb R}^d)^{k}} \eta(P_1,z-x_1)\cdots \eta(P_k,z-x_k)\cr
&\quad \times \Phi^{(j)}_{m_1,\dots,m_k}(P_1,\dots,P_k;d(x_1,...,x_k)){\mathbb Q}^k (d (P_1,...,P_k))\lambda (dz)\cr
&\ = \sum_{k=1}^\infty {(-1)^{k-1}}\int_{B}\overline\varphi^{(j)}(X_k,z)\lambda (dz).
\end{align*}
This shows the absolute continuity and, for the density $\overline\varphi^{(j)} (Z,\cdot)$, we get (almost surely)
\begin{align*}
\overline\varphi^{(j)} (Z,z)
&\ = \sum_{k=1}^\infty {(-1)^{k-1}}\overline\varphi^{(j)}(X_k,z).
\end{align*}
For $j=d$ and $j=d-1$, Corollary \ref{koro1} yields
\begin{align*}
\overline\varphi^{(d)} (Z,z)
&\ = \sum_{k=1}^\infty {(-1)^{k-1}}\frac{1}{k!} \overline V_d(X,z)^{k} = 1-\E^{-\overline V_d(X,z)}
\end{align*}
and
\begin{align*}
\overline \varphi^{(d-1)}(Z,z)
& = \sum_{k=1}^\infty {(-1)^{k-1}}\overline \varphi^{(d-1)}(X,z)\frac{\overline V_d(X,z)^{k-1}}{(k-1)!}\cr
& = \E^{-\overline V_d(X,z)}  \varphi^{(d-1)}(X,z).
\end{align*}
For $j<d-1$, we get from Corollary \ref{koro1}
\begin{align*}
\overline \varphi^{(j)}(Z,z)
& = \sum_{k=1}^\infty {(-1)^{k-1}}\sum_{s=1}^{(d-j)\wedge k} \frac{\overline V_d(X,z)^{k-s}}{(k-s)!} \cr
&\quad\times\sum_{m_1,\dots, m_s=j\atop m_1+\dots +m_d=(s-1)d+j}^{d-1} \frac{1}{s!}\
\overline \varphi_{ m_1,\dots,m_s}^{(j)}(X,\dots ,X;z,\dots ,z) \cr
& = \sum_{s=1}^{d-j}\sum_{r=0}^\infty \frac{{(-1)^{r+s-1}}}{r!}\overline V_d(X,z)^{r} \cr
&\quad\times\sum_{m_1,\dots, m_s=j\atop m_1+\dots +m_d=(s-1)d+j}^{d-1} \frac{1}{s!}\
\overline \varphi_{ m_1,\dots,m_s}^{(j)}(X,\dots ,X;z,\dots ,z) \cr
& = \E^{-\overline V_d(X,z)}\Biggl(\overline \varphi^{(j)}(X,z) -\sum_{s=2}^{d-j} \frac{{(-1)^{s}}}{s!} \cr
&\quad\times\sum_{m_1,\dots, m_s=j\atop m_1+\dots +m_d=(s-1)d+j}^{d-1} 
\overline \varphi_{ m_1,\dots,m_s}^{(j)}(X,\dots ,X;z,\dots ,z) \Biggr).
\end{align*}
This finishes the proof of the theorem.
\end{proof}

If $Z$ is stationary, then $X$ is stationary, the intensity function $\eta$ is a constant $\gamma >0$ and the densities in Theorem \ref{Poisson} are constants,
\begin{align*}
\overline \varphi_{ m_1,\dots,m_k}^{(j)}(X,\dots ,X)
=\gamma^k \int_{{\cal P}_0}\dots\int_{{\cal P}_0}\varphi^{(j)}_{m_1,\dots,m_k}(P_{1},\dots,P_{k})\,{\mathbb Q} (d P_1)\cdots \,{\mathbb Q} (d P_k) .
\end{align*}

\begin{koro}\label{koro2}
Let $Z$ be a stationary Boolean model in ${\mathbb R}^d$ with polytopal grains and let $\varphi$ be a local functional on ${\cal P}$. Then, 
\begin{equation*}
\overline \varphi^{(d)}(Z) = 1 -\E^{-\overline \varphi^{(d)}(X)},
\end{equation*} 
$$
\overline \varphi^{(d-1)}(Z) = \E^{-\overline \varphi^{(d)}(X)} \overline \varphi^{(d-1)}(X),
$$

and
\begin{eqnarray*}
\overline  \varphi^{(j)}(Z)&=&\E^{-\overline \varphi^{(d)}(X)}\left(\overline \varphi^{(j)}(X) -\sum_{s=2}^{d-j}\frac{(-1)^{s}}{s!}\right.\\
&&\hspace*{4mm}\times\,
\sum_{m_1,\dots, m_s=j+1\atop m_1+\dots +m_s=(s-1)d+j}^{d-1} \overline \varphi^{(j)}_{m_1,\dots ,m_s}(X,\dots,X)\Biggr),
\end{eqnarray*}
for $j=0,\dots ,d-2.$
\end{koro}

Let $\varphi$ be a local functional on ${\cal P}$ with local extension $\Phi$, let $X$ be a Poisson process on ${\cal P}$ with translation regular intensity measure and let $Z$ be the corresponding Boolean model. In this section and the previous one, we have introduced and studied the mean values $\overline\varphi (X,\cdot)$ and $\overline\varphi (Z,\cdot)$ as Radon-Nikodym derivatives of the expected random measures ${\mathbb E} \sum_{P\in X} \Phi(P,\cdot)$ respectively ${\mathbb E} \Phi(Z,\cdot)$ (under appropriate integrability assumptions). In this way, the densities $\overline\varphi (X,\cdot)$ and $\overline\varphi (Z,\cdot)$ depend on the choice of the local extension $\Phi$ and it is possible that different local extensions of $\varphi$ lead to different functions $\overline\varphi (X,\cdot)$ and $\overline\varphi (Z,\cdot)$. However, for stationary $X$ and $Z$ this cannot happen since then 
\begin{align*}
\overline \varphi (X)
=\gamma \int_{{\cal P}_0}\varphi(P)\,{\mathbb Q} (d P),
\end{align*}
as follows from the more general formula above. Also, $\overline\varphi (Z,\cdot)$ can be expressed by the corresponding densities of $X$ through Corollary \ref{koro2}. Moreover, in the stationary case, there are alternative density formulas by a limit procedure (see \cite{SW}, for the case of additive functionals).

\begin{theorem}\label{limit}
Let $X$ be a stationary Poisson process on ${\cal P}$, let $Z$ be the corresponding Boolean model, and let $\varphi$ be a local functional on ${\cal P}$. Then, 
\begin{align*}
\overline \varphi (X)
=\lim_{r\to\infty} \frac{1}{V_d(rW)}\ {\mathbb E}\sum_{P\in X} \varphi (P\cap rW)
\end{align*}
and
\begin{align*}
\overline \varphi (Z)
=\lim_{r\to\infty} \frac{1}{V_d(rW)}\ {\mathbb E}\ \varphi (Z\cap rW).
\end{align*}
\end{theorem}

\section{Applications}

In this section, we discuss some applications coming from particularly chosen functionals. 

Namely, we consider the special case $\varphi^{(j)} (P) = {\cal H}^j_{skel}(P), j=0,...,d$. We assume that $Z$ is a stationary and isotropic Boolean model with polytopal grains satisfying condition \eqref{integrability2}. In view of the stationarity, this integrability condition reads
\begin{align}\label{integrability4}
&\int_{({\cal P}_0)^k} \varphi^{(j)}_{m_1,\dots ,m_k}(P_1,\dots ,P_k)V_d(P_1+ C)\cdots V_d(P_k+ C)\,{\mathbb Q}^k (d (P_1,...,P_k))<\infty ,\quad\quad\end{align}
for all $j,k,m_1,...,m_k$ and compact $C\subset\R^d$. Here, it is sufficient, by the Steiner formula, to replace $C$ by the unit ball $B^d$. Then, the isotropy allows to replace $P_1,...,P_k$ by rotated versions $\vartheta_1P_1,...,\vartheta_kP_k$ and to integrate over all such rotations. From the results at the end of Section 5, we obtain that \eqref{integrability4} can be replaced by the simpler condition
\begin{align}\label{integrability5}
&\int_{{\cal P}_0} {\cal H}^j_{skel}(P)V_d(P+ B^d)\, {\mathbb Q} (d P)<\infty ,\quad\quad\end{align}
for $j=0,...,d$.

Applying Corollary \ref{koro2} to this situation, we get first the two obvious equations 
\begin{equation}\label{volumedens}
\overline {\cal H}^d(Z)= \overline V_d(Z) = 1 -\E^{-\overline V_d(X)}
\end{equation} 
and
\begin{equation}\label{surfacedens}
\overline {\cal H}^{d-1}_{skel}(Z) =  2\overline V_{d-1}(Z)=2\E^{-\overline V_d(X)} \overline V_{d-1}(X)=\E^{-\overline V_d(X)}\overline {\cal H}^{d-1}_{skel}(X).
\end{equation} 
For $j<d-1$, we obtain
\begin{eqnarray*}
\overline  {\cal H}^{j}_{skel}(Z)&=&\E^{-\overline V_d(X)}\left(\overline {\cal H}^{j}_{skel}(X) -\sum_{s=2}^{d-j}\frac{(-1)^{s}}{s!}\right.\\
&&\hspace*{4mm}\times\,
\sum_{m_1,\dots, m_s=j+1\atop m_1+\dots +m_s=(s-1)d+j}^{d-1} \overline \varphi^{(j)}_{m_1,\dots ,m_s}(X,\dots,X)\Biggr)
\end{eqnarray*}
where the mixed densities $ \overline \varphi^{(j)}_{m_1,\dots ,m_s}(X,\dots,X)$ are of the form
\begin{align*}
&\overline \varphi_{ m_1,\dots,m_s}^{(j)}(X,\dots ,X)\cr
&\ =\gamma^s \int_{{\cal P}_0}\dots\int_{{\cal P}_0}\varphi^{(j)}_{m_1,\dots,m_s}(P_{1},\dots,P_{s})\,{\mathbb Q} (d P_1)\cdots \,{\mathbb Q} (d P_s)\cr
&\ =\gamma^s \int_{({\cal P}_0)^s}\sum_{F_1\in{\cal F}_{m_1}(P_1)}\cdots\sum_{F_s\in{\cal F}_{m_s}(P_s)}[F_1,...,F_s]V_{m_1}(F_1)\cdots V_{m_s}(F_s)\cr
&\quad \times{\mathbb Q}^s (d (P_1, ...,P_s)).
\end{align*}
Again, we use the isotropy to replace $P_1,...,P_s$ by rotated versions $\vartheta_1P_1,...,\vartheta_sP_s$ and integrate over all  rotations. Then, we use \eqref{rotationformula} repeatedly and get
\begin{align*}
&\overline \varphi_{ m_1,\dots,m_s}^{(j)}(X,\dots ,X)\cr
&\ =\gamma^s \int_{({\cal P}_0)^s}\sum_{F_1\in{\cal F}_{m_1}(P_1)}\cdots\sum_{F_s\in{\cal F}_{m_s}(P_s)}c^d_j\prod_{i=1}^s c_d^{m_i}V_{m_1}(F_1)\cdots V_{m_s}(F_s)\cr
&\quad \times{\mathbb Q}^s (d (P_1, ...,P_s))\cr
&\ =\gamma^s \int_{({\cal P}_0)^s}c^d_j\prod_{i=1}^s c_d^{m_i}{\cal H}^{m_i}_{skel}(P_i)\, {\mathbb Q}^s (d (P_1, ...,P_s))\cr
&\ =c^d_j\prod_{i=1}^s c_d^{m_i}\overline {\cal H}^{m_i}_{skel}(X) .
\end{align*} 

Hence, we arrive at
\begin{align}\label{jdens}
\overline  {\cal H}^{j}_{skel}(Z)&=\E^{-\overline V_d(X)}\Biggl(\overline {\cal H}^{j}_{skel}(X)  \cr
& \quad -c^d_j\sum_{s=2}^{d-j}\frac{(-1)^{s}}{s!}
\sum_{m_1,\dots, m_s=j+1\atop m_1+\dots +m_s=(s-1)d+j}^{d-1} \prod_{i=1}^s c_d^{m_i}\overline {\cal H}^{m_i}_{skel}(X)  \Biggr)
\end{align}
for $j=0,...,d-2$.

The relations \eqref{jdens} are analogous to the ones for the intrinsic volumes and can be used, together with \eqref{volumedens} and \eqref{surfacedens}, to determine (estimate) the mean values $\overline {\cal H}^{i}_{skel}(X), i=0,...,d$, from (estimates of) the quantities $\overline {\cal H}^{j}_{skel}(Z), j=0,...,d$, on the left side. Simple unbiased estimators of the latter are given by
$$
\Phi^{(j)}(Z(\omega),W), j=0,...,d,
$$
where $\Phi^{(j)}$ is the $GP$-extension of the measure ${\cal H}^{j}_{skel}$ (for a $GP$-union $Q$ of polytopes, this is, in general, not the Hausdorff measure ${\cal H}^{j}$ on the $j$-skeleton of $Q$). In order to clarify the situation, we discuss the cases which are interesting for applications, namely $d=2$ and $d=3$.

For $d=2$, we have the three equations
\begin{align}\label{d2}
\overline A(Z) &= 1 -\E^{-\overline A(X)}\cr
\overline L(Z)&=\E^{-\overline A(X)}\overline L(X)\cr
\overline N_0(Z)&=\E^{-\overline A(X)}\left(\overline N_0(X)-\frac{1}{4\pi}\overline L(X)^2\right),
\end{align}
where $A$ denotes the area, $L$ the boundary length, and $N_0$ is the number of vertices of a polytope. The $GP$-extension of $N_0$, which is used on the left side, counts all convex vertices of a set in $U_{GP}({\cal P})$ positively and all concave vertices negatively. If the intensity $\gamma$ of $X$ is estimated by one of the traditional methods (e.g. through the equation for the Euler characteristic or with the tangent count), then $\gamma^{-1}\overline N_0(X)$ gives us the mean number of vertices of the typical polytope in $X$.

For $d=3$, we obtain
\begin{align}\label{d3}
\overline V(Z) &= 1 -\E^{-\overline V(X)}\cr
\overline S(Z)&=\E^{-\overline V(X)}\overline S(X)\cr
\overline L_1(Z)&=\E^{-\overline V(X)}\left(\overline L_1(X)-\frac{\pi^2}{32}\overline S(X)^2\right)\cr
\overline N_0(Z)&=\E^{-\overline V(X)}\left(\overline N_0(X)-\frac{1}{4\pi}\overline L_1(X)\overline S(X) + \frac{\pi}{384}\overline S(X)^3\right).
\end{align}
Here, $V$ is the volume, $S$ the surface area, $L_1$ the total edge length, and $N_0$ the number of vertices. The $GP$-extensions of the latter two quantities are as follows. For $L_1$, all convex edges count positively and all concave edges count negatively. The vertices of a $GP$-union are of three kinds; convex vertices (these are vertices of one of the original polytopes), concave vertices (which arise as the intersection of three facets) and saddlepoint-like vertices (which come from intersections of edges of one polytope with facets of another). The convex and concave vertices are counted positively and the saddlepoint-like vertices are counted negatively.
Again, if $\gamma$ is obtained by one of the classical methods, the formulas can be used to get the mean number of vertices and the mean total edge length of the typical polytope in $X$. 

We remark that the formula for $\overline L_1(Z)$ has a much more intuitive meaning than the corresponding result in the classical case of intrinsic volumes. There, the corresponding formula concerns the specific integral mean curvature $\overline M(Z)$  of the Boolean model $Z$.

\section{Polyconvex grains}

Formulas for the mean total Hausdorff measure of the $j$-skeleton of the Boolean model $Z$ could also be obtained by applying the Slyvniak-Mecke formula (see e.g. Corollary 3.2.3 in \cite{SW}) to the individual $j$-faces of the grains, more precisely to the parts of the $j$-faces of a given grain of $X$ which are not covered by other grains. The resulting expressions will differ from those in \eqref{jdens} since they concern the mean total $j$-dimensional content of the visible (e.g. uncovered) parts of the $j$-faces of the grains, whereas \eqref{jdens} uses the mean values of the $GP$-extensions of the Hausdorff measure. This means, in particular, that the information contained in the non-convex boundary parts of $Z$ (like concave edges or saddlepoint-like vertices in the three-dimensional case) would not be used in this approach. Nevertheless, it would be interesting to compare the resulting formulas in both approaches to separate the information coming from these non-convex parts of $Z$. This will be part of an investigation which is in progress.

However, the advantage of the approach presented here becomes apparent if Boolean models $Z$ with polyconvex grains are considered. To be more precise, we now assume that the grains of $Z$ (i.e. the particles in the Poisson process $X$) are themselves $GP$-unions of convex polytopes. As we will indicate now, the results obtained so far hold true (under suitable integrability conditions) for Boolean models with such poly-polytopal grains. In this situation, the approach based on the Slyvniak-Mecke formula would only work if the non-convex boundary parts in $Z$ coming from the union of particles can be distinguished from the non-convex boundary parts of the individual particles, a condition which is rarely satisfied in practice.

Concerning the extension of the results to poly-polytopal grains, we first observe that, given a local functional $\varphi$ on $\cal P$ with local extension $\Phi$, the translative integral formulas from Theorem \ref{trans1} and Corollary \ref{trans2} hold true for sets $P,Q,P_1,...,P_k\in U_{GP}({\cal P})$. More precisely, if $P=\bigcup_{i=1}^l P_i, Q=\bigcup_{k=1}^n Q_k$ are standard representations of $P,Q\in U_{GP}({\cal P})$, then, for $\lambda$-almost all $x$, $P\cap Q^x\in U_{GP}({\cal P})$ and $P\cap Q^x = \bigcup_{i=1}^l\bigcup_{k=1}^n (P_i\cap Q_k^x)$ is a standard representation. For Borel sets $A,B\in{\cal B}$, we therefore obtain from the inclusion-exclusion formula \eqref{i-e-formula2} and Theorem \ref{trans1}, that
\begin{align*}
\int _{{\mathbb R}^d}& \Phi^{(j)}(P \cap Q^x, A \cap B^x) \, \lambda(d x)\cr
 &= \sum_{r=1}^l\sum_{s=1}^n(-1)^{r+s}\sum_{1\le i_1<\dots < i_r\le l}\ \sum_{1\le j_1<\dots < j_s\le n}\cr
 &\quad\int _{{\mathbb R}^d} \Phi^{(j)}((P_{i_1}\cap\cdots\cap P_{i_r}) \cap (Q_{j_1}\cap\cdots\cap Q_{j_s})^x, A \cap B^x) \, \lambda(d x)\cr
 &=\sum_{r=1}^l\sum_{s=1}^n(-1)^{r+s}\sum_{1\le i_1<\dots < i_r\le l}\ \sum_{1\le j_1<\dots < j_s\le n}\cr
 &\quad \sum_{m=j}^d \Phi_{m,d-m+j}^{(j)} (P_{i_1}\cap\cdots\cap P_{i_r},Q_{j_1}\cap\cdots\cap Q_{j_s};A \times B) \cr
 &=\sum_{m=j}^d \Phi_{m,d-m+j}^{(j)} (P,Q;A \times B)
\end{align*} 
where the finite signed measure $\Phi_{m,d-m+j}^{(j)}(P,Q;\cdot)$ is defined by
\begin{align*}
 &\Phi_{m,d-m+j}^{(j)} (P,Q;\cdot) = \sum_{r=1}^l\sum_{s=1}^n(-1)^{r+s}\cr
 &\ \sum_{1\le i_1<\dots < i_r\le l}\ \sum_{1\le j_1<\dots < j_s\le n}
  \Phi_{m,d-m+j}^{(j)} (P_{i_1}\cap\cdots\cap P_{i_r},Q_{j_1}\cap\cdots\cap Q_{j_s};\cdot)  .
  \end{align*}
Notice that the formulas
\begin{align*}
&  \Phi_{j,d}^{(j)} (P,Q;A \times B) = \Phi^{(j)}(P,A) \lambda (Q\cap B), \\
&  \Phi_{d,j}^{(j)}(P,Q;A \times B) = \lambda (P\cap A) \Phi^{(j)}(Q,B)
\end{align*}
still hold. The iterated translative integral formula \eqref{mixedint} for poly-polytopal sets follows in a similar way, as do the two translative integral formulas from Corollary \ref{trans2}. The essential properties of the mixed measures and functionals carry over from polytopes to $GP$-unions, namely homogeneity, symmetry and decomposability.

The results on Poisson processes in Section 7 carry over to processes on ${\cal U}_{GP}({\cal P})$, if we replace the integrability condition \eqref{integrability2} by the stronger condition
\begin{align}\label{integrabilitynew}
&\int_{({\cal P}_0)^k}\int_{({\mathbb R}^d)^k} \|\Phi^{(j)}_{m_1,\dots ,m_k}(P_1,\dots ,P_k;\cdot)\|{\bf 1}\{P_1^{x_1}\cap C\not=\emptyset\} \eta (P_1,x_1)\cdots \cr
&\times {\bf 1}\{P_k^{x_k}\cap C\not=\emptyset\} \eta (P_k,x_k)\,\lambda^k (d(x_1,...,x_k))\,{\mathbb Q}^k (d (P_1,...,P_k))<\infty ,\quad\quad\end{align}
for any compact $C\subset\R^d$, where $\|\mu \|$ denotes the total variation norm of the measure $\mu$. Similarly, the results on Boolean models from Section 8 hold true, for poly-polytopal grains, if in the assumptions \eqref{integrability2} is replaced by \eqref{integrabilitynew}. Thus also the formulas \eqref{d3} and \eqref{d2} are valid for poly-polytopal grains and the interpretation (respectively, the evaluation) of the left sides remains the same. On the ride side, one has to observe that $\overline L_1(X)$ describes no longer the mean total edge length of the typical poly-polytopal grain of $X$ (if the intensity $\gamma$ is known) but the $GP$-extension of the edge length, counting convex edges positively and concave edges negatively. Similarly, $\overline N_0(X)$ is $\gamma$ times the $GP$-extended mean vertex number of the typical grain of $X$.

\end{document}